\documentclass[12pt]{amsart}
\usepackage{xcolor} 
\usepackage{graphicx}
\usepackage{enumerate}
\usepackage[mathscr]{euscript}

\usepackage{amsmath}
\usepackage{amssymb}
\usepackage{amsrefs}

\DeclareMathSymbol\HH 0{AMSb}{`H}
\DeclareMathSymbol\I  0{AMSb}{`I}
\DeclareMathSymbol\R  0{AMSb}{`R}

\def\q{\mathbf{q}}

\theoremstyle{plain}
\newtheorem{theorem}{Theorem}[section]
\newtheorem{lemma}[theorem]{Lemma}
\newtheorem{proposition}[theorem]{Proposition}
\newtheorem{corollary}[theorem]{Corollary}
\newtheorem{definition}[theorem]{Definition}

\theoremstyle{definition}

\newtheorem{question}{Question}[section]

\theoremstyle{remark}
\newtheorem{remark}{Remark}

\newtheorem{fact}{Fact}

\newtheorem{claim}[remark]{Claim}

\numberwithin{equation}{section}

\DeclareMathOperator{\dom}{dom}
\DeclareMathOperator{\supp}{supp}

\DeclareMathOperator{\Fn}{\operatorname{Fn}}
\DeclareMathOperator{\val}{\operatorname{val}}

\title
{S-spaces and large continuum}
\author[A. Dow]{Alan Dow}
\address{University of North Carolina at Charlotte, 
Charlotte, NC 28223}
\email{adow@uncc.edu}
\author[S. Shelah]{Saharon Shelah}
\address{Department of Mathematics, Rutgers University, Hill Center,
 Piscataway, 
 New Jersey, U.S.A. 08854-8019}
\curraddr{Institute of Mathematics\\Hebrew University\\
Givat Ram, Jerusalem 91904, Israel}
\email{shelah@math.rutgers.edu}

\date{\today}

\thanks{
The research of the second  author was
 supported by
   the United States-Israel Binational Science Foundation (BSF Grant
   no. 2010405), and by the 
 NSF grant No. NSF-DMS 1101597. This is manuscript number F2103.}



\begin{document}
\begin{abstract}
 We prove that it is consistent with large values of the continuum
 that there are no S-spaces. We also show that we can
 also have that compact separable spaces
 of countable tightness have cardinality at most the continuum.
 \end{abstract}
\maketitle

\bibliographystyle{plain}

\section{Introduction}
 
 An S-space is a regular hereditarily separable space
 that is not Lindel\"of. If an S-space exists it can be assumed
 to be a topology on $\omega_1$ in which initial segments
 are open
 \cite{SandL}. The continuum hypothesis implies that
 S-spaces exist \cite{HJ73}
 and the existence of a Souslin tree
 implies that S-spaces exist \cite{RudinS}.
  Therefore it is consistent
 with any value of $\mathfrak c$ that S-spaces exist.
  Todorcevic \cite{stevoSspace} proved
  the major result that it is 
  consistent with $\mathfrak c=\aleph_2$ that
  there are no S-spaces. He also remarks that
 this   follows from PFA.  
  We prove that it is consistent with arbitrary large values
  of $\mathfrak c$ that there are no S-spaces. 
  Our method adapts the approach used 
   in \cite{stevoSspace} and incorporates
   ideas, such as \textit{the Cohen real trick} in Lemma \ref{Cohengeneric},
    first introduced in  \cites{Avraham,ARS}. 
   
   The outline of the proof (of Theorem \ref{mainthm})
   is that we choose a
   regular cardinal $\kappa$ in a model of GCH. 
   We construct a preparatory mixed 
   support iteration sequence
    $\langle P_\alpha, \dot Q_\beta : \alpha \leq\kappa,
     \ \beta <\kappa\rangle$ consisting of iterands
     that are Cohen posets and 
     cardinal preserving subposets
     of Jensen's poset for adding a generic cub.
Following methods first introduced in \cite{Mitchell},
 but more closely those of \cite{stevoSspace}, 
 the poset $P_\kappa$ is shown to be cardinal 
  preserving.
     We then extend the iteration sequence to one
     of length $\kappa+\kappa$ with iterands 
     that are ccc posets of cardinality 
     less than $\kappa$.
     These iterands are the same as those used
     in \cite{stevoSspace}. For cofinally many $\beta<\kappa$, 
      $\dot Q_{\kappa+\beta}$ is constructed so as
     to add an uncountable discrete subset to
     a $P_\beta$-name of an S-space. The bookkeeping
     is routine to ensure that $P_{\kappa+\kappa}$ forces
     there are no S-spaces. The challenging part of
     the proof is to prove that these $\dot Q_\beta$ 
     ($\kappa \leq \beta < \kappa+\kappa$)
     are ccc in this new setting. In the final section,
      we use similar techniques to produce a model
      in which compact separable spaces of countable tightness
       have cardinality at most $\mathfrak c$.

 \section{Constructing $P_{\kappa}$}
  
  Throughout the paper we assume that GCH holds and
  that $\kappa>\aleph_2$ is a regular uncountable cardinal.
 
\begin{definition}
 The Jensen poset $\mathscr J$ is the set of pairs
  $(a,A)$ where $a$ is a countable closed subset of $\omega_1$
  and $A\supset a $ is an uncountable closed subset of $\omega_1$. 
  The condition   $(a,A)$ is an extension of $(b,B)\in \mathscr J$ providing
 $a$ is an end-extension of   $b$   and  $A\subset B $.
 \end{definition}

   We use $\mathbf{E}$ to denote the set $\{ \lambda+2k :
    \lambda <\kappa \ \mbox{a limit} , k\in\omega\}$. 
    We also choose a family  
  $\mathscr I =
 \{ I_\gamma : \gamma\in \mathbf{E}\}$  
 of subsets of $\kappa$ such that, for each $\mu<\gamma\in \mathbf{E}$
   
\begin{enumerate}
  \item 
  $   \gamma\in  I_\gamma\subset \gamma+1$
   and $|I_\gamma| \leq \aleph_1$,
\item if $\gamma<\omega_2$, then $I_\gamma = \gamma+1$,
 \item if $\mu\in I_\gamma\cap\mathbf{E}$, then 
 $I_\mu\subset I_\gamma$ 
 \item for all $I\in [\kappa]^{\aleph_1}$, 
 the set $\{ \gamma :  I \subset I_\gamma\}$ is
  unbounded in $\kappa$.
\end{enumerate}

 \noindent Say that a set $I\subset \kappa$ is $\mathscr I$-saturated
 if it satisfies that $I_\mu\subset I$ for all $\mu\in I\cap\mathbf{E}$.
 Of course, each $I_\gamma\in \mathscr I$ is 
  $\mathscr I$-saturated.
 
\begin{definition}

\textbf{A.}\ 
 We define a mixed\label{conditions}
  support iteration sequence $\langle
  P_{\alpha },\dot Q_{\beta } : \alpha \leq \kappa,\ \ 
  \beta <\kappa\rangle$:
\begin{enumerate}
 \item $P_{ 0}=\emptyset$,
 \item $p\in P_{\alpha }$ is a function with $\dom(p)$,
 a countable subset of $  \alpha$, such that
  $\dom(p)\cap \mathbf{E}$ is finite, 
  \item for all $p\in P_{\alpha }$ and $\beta\in \dom(p)$,
    $p(\beta)$ is a $P_\beta$-name forced by $1_{P_\beta}$ 
    to be an element of $\dot Q_\beta$,
  \item the support of a $P_{\alpha }$-name
   $\tau$, $\supp(\tau)$, is defined,
  by recursion on $\alpha$ 
   to
  be the union of the
  set $\{  \supp(\sigma) \cup \dom(q) : ( \sigma,q)\in \tau\}$,
  \item for $\alpha\in \mathbf{E}$, $\dot Q_{\alpha}$ is 
  the trivial
  $P_{\alpha}$-name for $\mathcal C_{\omega_1} = \Fn(\omega_1,2)$
  (i.e. each element of $\dot Q_{\alpha}$ has empty support),
  \item for $\alpha\in \mathbf{E}$, $\dot Q_{\alpha+1 }$ is the 
  subposet of the standard $P_{\alpha+1}$-name for $\mathscr J$ 
  consisting of the $P_{\alpha+1}$-names that are forced to
  have  the form $(\dot a,\dot A)$  
   where  $\supp(\dot a) \subset \mathbf{E}\cap I_\alpha$,
  $ \supp(\dot A)\subset \alpha$,
   and $1_{P_\alpha+1} $
   forces that $(\dot a, \dot A)\in \mathcal J$.
 $\dot Q_{\alpha+1}$ is chosen so as to be sufficiently rich in names
 in the sense that if $p\in P_{\alpha+1}$ and 
 $\dot q$ is a $P_{\alpha+1}$-name such 
 that $p\Vdash_{P_\alpha}\dot q\in \dot Q_{\alpha+1}$, 
 then there is a $\dot q_1\in \dot Q_{\alpha+1}$
 such that $p\Vdash \dot q = \dot q_1$.
\end{enumerate}
\textbf{B.}\ 
For each $\alpha\in \mathbf{E}$, we let $\dot C_\alpha$ denote
the $P_{\alpha+2}$-name of
the generic subset of $\omega_1$  added by $\dot Q_{\alpha+1}$.  
 \end{definition}

\begin{remark}
 Since we defined   the family $\mathscr I$ to have
 the property that $I_\gamma  = \gamma +1$ for
  all $\gamma\in \omega_2\cap \mathbf{E}$,
  it follows that
  $P_{\omega_2}$ is isomorphic
  to that used in \cite{stevoSspace}. It also
  follows that for all $\beta\in\omega_2\cap \mathbf{E}$,
   $P_{\beta+1}\Vdash \dot Q_{\beta+1}$ is countably closed. 
   We necessarily lose this property for $\omega_2
   \leq \beta$ for any family $\mathscr I$ satisfying
   our properties (1)-(4). Nevertheless,  our 
   development of the properties of $P_\kappa$
   will closely follow that of \cite{stevoSspace}.
\end{remark}

\begin{remark}  We prove in Lemma \ref{itiscub} that,
for each $\alpha\in \mathbf{E}$, 
 $\dot C_\alpha$   is forced, as hoped,
 to be a cub. However, 
even though, for $\beta\geq\omega_2$,
  $P_{\beta+1}$ 
does not force that $\dot Q_{\beta+1}$ is countably 
closed,  we make note of subsets of
the iteration sequence that have special properties, such as
in Lemma \ref{desc}.
\end{remark}

 For any ordered pair $(a,b)$, 
 let $\pi_0(  (  a,b) ) =   a$ and $\pi_1((a,b)) = b$. 
 For convenience, for an element $v$ of $V$ and any $\alpha<\kappa$,
  we identify the usual trivial $P_\alpha$-name for $v$ with $v$ itself.
  In particular, if $s\in\mathcal C_{\omega_1}$ and $\alpha\in 
  \mathbf{E}$, then $s\in \dot Q_{\alpha}$. Similarly, 
  if $(\dot a,\dot A)$ is a pair of the form specified in 
  Definition \ref{conditions}(6), then again
    $(\dot a, \dot A)$ can be regarded as  an element of $\dot Q_{\alpha+1}$.
 We will say that a $P$-name $\tau$ for a subset of an
    ordinal $\lambda$ and poset $P$ is canonical if
    it is a subset of $\lambda\times P$ (and optionally,
    if  $\{ p : (\alpha,p) \in \tau\}$ is an antichain
    for all $\alpha\in\lambda$). 
      Let $\mathscr D_\beta$ denote the set of 
    canonical $P_{\beta}$-names of closed and unbounded
    subsets of $\omega_1$.

\begin{definition} For each $\alpha <\kappa$, let 
 $P_\alpha'$  denote the 
subset of $P_\alpha$, where $p\in P_\alpha'$ providing
for all $\beta\in \dom(p)\cap \mathbf{E}$,  $p(\beta)$ is,
literally,  an element of $\mathcal C_{\omega_1}$.
\end{definition}

\begin{lemma} For all $\alpha\leq \kappa$,  $P_\alpha'$ is
 a dense subset of $P_\alpha$.
\end{lemma}

\begin{proof}   Assume $\alpha \leq\kappa$ and that, by
induction, $P_\beta'$ is a dense subset of $P_\beta$
for all $\beta < \alpha$. Consider any $p\in P_\alpha$.
If $\alpha$ is a limit, choose any $\beta <\alpha$ such
  that $\dom(p)\cap \mathbf{E}\subset\beta$. Choose
  any $p'\in P_{\beta}'$ so that $p'<p\restriction \beta$. 
  We then have that $p'\cup p\restriction (\alpha\setminus\beta)$
  is a condition in $P_\alpha$ that is below $p$.
  
  Now let $\alpha = \beta+1$.
   If $\beta\in\mathbf{E}$,
   then choose $p'\in P_\beta'$ so that
   there is an $s\in \mathcal C_{\omega_1}$
   such that $p'\Vdash_{P_\beta}p(\beta) = s$.
   Then the desired extension of $p$ in 
    $P_\alpha'$ is $p'\cup \langle \beta,s\rangle$.
    Similarly, if  
   $\beta\notin\mathbf{E}$ and $p'\in P_{\beta}'$
    with $p'<p\restriction\beta$,
   then $p'\cup \langle \beta,p(\beta)\rangle\in P_\alpha'$. 
\end{proof}

\begin{proposition}
 If $p\in P_{\kappa}$\label{shrink}  
 then for every $I\subset \kappa$,
 $p\restriction I\in P_{\kappa}$ and $p\leq p\restriction I$.
 \end{proposition}

\begin{definition}
 For\label{moresubs} a subset $I\subset \kappa$ and $\alpha\leq \kappa$,
  let $P_{\alpha   }(I)$ denote the subset
   $\{ p\in P'_{\alpha } : \dom(p)\subset I\}$. 
   \end{definition}
   
     Recall that for   posets $(P,<_P)$
 and $(R,<_R)$, $P$ is a complete
 subposet  of $R$, i.e. $P\subset_c R$, providing
 
\begin{enumerate}
 \item $P\subset R$, $<_P = <_R\cap (P\times P)$,
 \item $\perp_P = \perp_R\cap (P\times P)$, where $\perp$ is the incompatibility relation,
  \item for each $r\in R$, the set of projections,
   $\mbox{proj}_P(r) $, is not empty, where
    $\mbox{proj}_P(r) =
    \{ p \in P : (\forall q\in P) (q<_P p \Rightarrow q\not\perp_R r)\}$.
 \end{enumerate}
 If $P\subset_c R$, then $R/P$ is often used to denote
 the $P$-name of the poset satisfying that
  $R \simeq P\star R/P$. In fact,
  $R/P$ can be defined so that simply
  if  $G\subset P$ is a generic
  filter, then $\val_G(R/P) = \{ 
   r\in R : \mbox{proj}_P(r)\cap G\neq\emptyset\}$  
   with the ordering inherited from $<_R$.
   With this view,  $\val_G(R/P) = G^+$ where,
    as is standard, $G^+ =\{ r\in R : 
    (\forall p\in G) r\not\perp p\}$.
    Of course it follows that for $\beta < \alpha \leq\kappa$,
     $P_\beta \subset_c P_\alpha$.

It is clear that  $P_{\alpha}(\mathbf{E})$ is isomorphic
to (the usual dense subset of) a finite support iteration of 
the Cohen poset $\mathcal C_{\omega_1}$.

\begin{proposition}
 For each $\alpha\leq\kappa$, the set\label{justCohen}
    $P_{\alpha }(\mathbf{E})\subset_c P_\alpha$  and  is ccc.
\end{proposition}

 \begin{definition} For each $\alpha\in \mathbf{E}$, 
 let $Q'_{\alpha+1}$ be the subset of $\dot Q_{\alpha+1}$
 consisting of those pairs $(\dot a,\dot A)$ as in Definition \ref{conditions}(6).
 \end{definition}

We may note that, for each $(\dot a,\dot A)\in  Q_{\alpha+1}'$,
  $\dot a$ is a $P_{\alpha+1}(I_\alpha\cap \mathbf{E})$-name
  and $\dot A$ is a $P_\alpha$-name that is forced by $1_{P_\alpha}$ to be a cub
  subset of $\omega_1$. Also, for every $p\in P_{\alpha+1}$,
   $p\restriction\alpha \Vdash p(\alpha+1)\in Q_{\alpha+1}'$.
 
 \begin{lemma} If $\alpha\in \mathbf{E}$ and\label{desc},
 and 
  $\{ (\dot a_n, \dot A_n) : n\in \omega\}\subset Q_{\alpha+1}'$ is
  a sequence that  satisfies,
  for each $n\in\omega$,
   $1\Vdash_{P_{\alpha+1}} (\dot a_{n+1},\dot A_{n+1})\leq (\dot a_n,\dot A_n)$,
   then there is a condition $(\dot a,\dot A)\in Q_{\alpha+1}'$
   such that 
   \begin{enumerate}
   \item $1\Vdash_{P_{\alpha+1}} $ forces that $\dot a$ is the closure of $ \bigcup\{\dot a_n : n\in\omega\}$,
   \item $1_{P_\alpha}$ forces that $\dot A $ equals $\bigcap \{ \dot A_n : n\in\omega\}$,
   \item $1\Vdash_{P_{\alpha+1}}$ forces that $(\dot a,\dot A) = \bigwedge\{
   (\dot a_n, \dot A_n) : n\in \omega\}$.
   \end{enumerate} 
   \end{lemma}

\begin{proof}  In the forcing extension by a $P_{\alpha+1}$-generic
filter $G$, it is clear
that $(\mbox{cl}\left(\bigcup\{  \val_G(\dot a_n)\right),
   \bigcap \{\val_G(\dot A_n) : n\in \omega\})$ is the meet in 
     $\mathcal J$ of the sequence
       $\{ (\val_G(\dot a_n), \val_G(\dot A_n)) : n\in \omega\}$.
       We just have to be careful about the supports of the names for these
       objects.
Each $\dot a_n$ is a $P_{\alpha+1}(I_\alpha)$-name and so 
it is clear that there is a $P_{\alpha+1}(I_\alpha\cap \mathbf{E})$-name,
 $\dot a$, 
such that $1\Vdash_{P_{\alpha+1}} \dot a = 
\mbox{cl}\left(\bigcup\{\dot a_n : n\in \omega\}\right)$. This is the only
subtle point. Any $P_\alpha$-name, $\dot A$, for $\bigcap\{ \dot A_n : n \in\omega\}$
is adequate (although we are using that each $\dot A_n$ is a 
 $P_\alpha$-name forced by $1$ to be a cub).
\end{proof}

When we have a sequence $\{(\dot a_n, \dot A_n) :n\in\omega\} \subset
 \dot Q_{\alpha+1}'$ as in the hypothesis of Lemma \ref{desc}, we
 will use $\bigwedge\{ (\dot a_n, \dot A_n): n\in \omega\}$ to denote
 the element $(\dot a,\dot A)$ in the conclusion of the Lemma.

Let $<_E$ denote the relation on $P_{\kappa}$
defined by $p_1 <_{ E} p_0$ providing 
\begin{enumerate}
\item  
$p_1 \leq p_0$,
\item
$p_1\restriction \mathbf{E} = p_0\restriction \mathbf{E}$,
\item for $\beta\in \dom(p_0)$, 
  $\mathbf{1}_{P_\beta} \Vdash p_1(\beta) < p_0(\beta)$.
\end{enumerate}

For $r\in P_{\kappa }(\mathbf{E}) $ 
and compatible $p\in P_{\kappa}$,
 let $p\wedge r$ denote the condition with domain
  $\dom(p)\cup \dom(r)$ satisfying $(p\wedge r)(\beta)
   = p(\beta)\cup r(\beta)$ for  $\beta\in \dom(r)$
   and $ (p\wedge r)(\beta) = p(\beta)$ for
    $\beta\in \dom(p)\setminus \dom(r)$. 
    For convenience, let $p\wedge r$ equal $p$
    if $r\in P_{\kappa}$ is not compatible with $p$.

\begin{lemma}
 Assume that $\{ p_n : n\in\omega\}\subset P'_\kappa$
 is a $<_E$-descending sequence. Then there\label{omega}
 is a $p_\omega\in P_\kappa'$ such
 that $\dom(p)=\bigcup_n \dom(p_n)$
 and 
 $p_\omega<_E p_n$ for all $n\in\omega$.
\end{lemma}

\begin{proof}
We let $J = \bigcup\{\dom(p_n) : n\in \omega\}$.
 We define $p_\omega\restriction \beta$ by induction
 on $\beta\in \mathbf{E}$ so that
  $\dom(p_\omega\restriction\beta) = J\cap \beta$.
  For limit $\alpha$, simply $p_\omega\restriction\alpha
   =\bigcup_{\beta<\alpha} p_\omega\restriction\beta$. 
   If $p_\omega\restriction \beta <_E p_n\restriction\beta$
   for all $n\in\omega$ and $\beta <\alpha$,
    then we have $p_\omega\restriction\alpha <_E p_n\restriction
    \alpha$ for all $n\in\omega$. 
    Now let $\alpha = \beta+2$ with $\beta\in \mathbf{E}$
    and assume that we have defined
     $p_\omega\restriction \beta$ as above. 
     If $\beta\in J$, then let $p_\omega(\beta) = p_0(\beta)$.
     If $\beta+1\in J$, then 
      $\mathbf{1}_{P_{\beta+1}}$
      forces that $\{ p_n(\beta+1) :  n\in\omega\}$ 
      is a descending sequence in $\dot Q_{\beta+1}$.
      We define $p_\omega(\beta+1)$
      to equal $\bigwedge \{ p_n(\beta+1) : n\in\omega\}$. 
      It follows by the definition of
    $\bigwedge \{ p_n(\beta+1) : n\in\omega\}$,
    that $\mathbf{1}_{P_{\beta+1}}\Vdash
     p_\omega(\beta+1) < p_n(\beta+1)$ for all $n\in\omega$.  
\end{proof}

\begin{lemma}
 For\label{1proper}
  every $p_0\in P'_{\kappa }$ and dense 
 subset $D$ of $P_{\kappa }$, there is a
  $p <_E p_0$ satisfying that
  the set $D\cap \{ p\wedge r : r\in P_{\kappa }(\mathbf{E})\}$
  is predense below $p$. Moreover,
   there is a countable subset of
   $D\cap \{ p\wedge r : r\in P_{\kappa}(\mathbf{E})\}$
   that is predense below $p$.
\end{lemma}

\begin{proof}
 Let $r_0 = p_0\restriction \mathbf{E}$.
 There is nothing to prove if $p_0\in D$ so assume
 that it is not. 
 By induction on $0<\eta<\omega_1$, we choose,
 if possible, 
 conditions $p_\eta, r_\eta $  
 such that,
 for all $\zeta <\eta$: 
\begin{enumerate}
 \item
  $p_\zeta <_E p_\eta$ and $r_\zeta < r_0$,
 \item $p_{\zeta}\wedge r_{\zeta}\in D$,
  \item  $(p_\eta\wedge r_\eta)\perp (p_\zeta\wedge r_\zeta)$.
   \end{enumerate}
   Suppose
   that we have so chosen $\{ p_\zeta, r_\zeta : \zeta < \eta\}$. 
   Let $L_\eta = \bigcup \{ \dom(p_\zeta) :
    \zeta < \eta\}$. 
    If $\eta = \beta+1$, let $\bar p_\eta = p_\beta$. 
    If $\eta$ is a limit, then let $\bar p_\eta$ be a condition
    as in Lemma \ref{omega} for some cofinal sequence
    in $\eta$.
     If $\{ p_\zeta\wedge r_\zeta :  \zeta < \eta\}$
      is predense below $\bar p_\eta$, we halt the induction
      and set $p = \bar p_\eta$. 
      Otherwise we choose any $p_\eta <_E \bar p_\eta$
      and an $r_\eta\supset r_0$ so that
       $p_\eta\wedge r_\eta$ 
  in $D$. The induction will halt for
      some $\eta<\omega_1$ since the family
       $\{ r_\zeta  : \zeta<\eta\}$ 
       is evidently an antichain in 
        $P_\kappa(\mathbf{E})$. 
 \end{proof}

\begin{corollary} For each $\beta\in \mathbf{E}$, 
 $P_\beta$ is proper\label{noreals}
  and $P_\beta/P_\beta(\mathbf{E}\cap \beta)$
 does not add any reals. 
\end{corollary}

\begin{proof}
Let $P_\beta\in M$ where $M$ is a countable
elementary submodel of $H(\kappa^+)$. 
Let $\{ D_n : n\in \omega\}$ be an enumeration 
of the dense open subsets of $P_\beta$ that are
members of $M$. By Lemma \ref{1proper}, 
 we have that for each $q\in P_\beta\cap M$ and
  $n\in\omega$, there is a $\bar q<_E q$ also in 
   $P_\beta\cap M$ so that 
    $D_n\cap  \{ \bar q \wedge r : r\in 
     P_\beta(\mathbf{E})\cap M\}$ is
     predense below $\bar q$. 
     Let $M\cap \omega_1=\delta$. 
     Fix any $p_0\in P_\beta\cap M$.
  By a simple recursion, we may construct
  a $<_E$-descending sequence
   $\{ p_n : n\in\omega\}\subset M$
   so that, for each $n$,
    $D_n\cap \{ p_{n+1}\wedge r : r\in 
     P_{\beta}(\mathbf{E})\cap M\}$ is
      predense below $p_{n+1}$.
      By Lemma \ref{omega}, we have
      the $(P_\beta,M)$-generic condition $p_\omega$. 
      It is clear that for each $P_\beta$-name $\tau\in M$
      for a subset of $\omega$, $p_\omega$ forces
      that $\tau$ is equal to a $P_\beta(\mathbf{E})$-name.
      This implies that $P_\beta/P_\beta(\mathbf{E}\cap \beta)$
      does not add reals.
\end{proof}
 
 We can now prove that $P_{\beta+2}$ does indeed
 force that $\dot C_\beta$ is a cub.

\begin{lemma}
 For each $\beta\in \mathbf{E}$, 
  $P_{\beta+2}$ forces\label{itiscub}
   that $\dot C_\beta$ is unbounded in
   $\omega_1$. 
\end{lemma}
 
\begin{proof}
Let $p\in P_{\beta+2}$ be any condition and
let $\gamma\in \omega_1$. By possibly strengthening
 $p $ we can assume that $p(\beta+1)\in Q'_{\beta+1}$. 
We find 
 $q<p$ so that $q\Vdash \dot C_\beta\setminus\gamma$
 is not empty. Let $p, P_{\beta+2}$ be members of
  a countable elementary submodel $M\prec
   H(\kappa^+)$. Let $\bar p< p\restriction 
   \beta$ be $(P_\beta,M)$-generic  
   and let $\dot D = \pi_1(p(\beta+1))\in 
    \mathscr D_\beta$.  
    Since $p, \dot D$ are members of $M$
      and $p$ forces that $\dot D$ is a cub, it
      follows that $\bar p\Vdash \delta\in \dot D$. 
      It also follows that $\bar p \Vdash \dot a\subset 
      \dot D\cap \delta$.
      Let $\dot a_1$ be the  $P_{\beta+1}$-name
 that has support equal to the support of the
 name $\dot a$ and satisfies that 
  $\mathbf{1}_{P_\beta+1}\Vdash \dot a_1 = \dot a\cup\{ \delta\}$.
      Let $\dot E $ be the $P_\beta$-name
  for $\dot D\cup \{\delta\}$ and notice
  that, given that $(\dot a,\dot D)\in  Q'_{\beta+1}$,
  we have that 
   $(\dot a_1, \dot E)$ is also in $ Q'_{\beta+1}$.
Now  let $q\in P_{\beta+2}$
  be defined according to $ q\restriction \beta = \bar p$,
   $q(\beta) = p(\beta)$,
    and $q(\beta+1) =  (\dot a_1, \dot E)$.
      It is immediate  that $q\restriction \beta+1
       < p\restriction\beta+1$. 
       Also, $q\restriction \beta+1$ forces
       that $\dot a$ is an initial segment of $\dot a_1$,
       that $\dot a_1\subset \dot D$, and
       that $\dot E\subset \dot D$.  
       Therefore, $q<p$ and $q\Vdash \delta\in \dot C_\beta$.
\end{proof}

\begin{lemma}
 For each $\beta \leq \kappa$, $P_\beta$ 
 satisfies\label{noreals} the $\aleph_2$-cc.
\end{lemma}

\begin{proof}
We prove the lemma by induction on $\beta$. 
If $\beta\in \mathbf{E}$ and $P_\beta$ satisfies
the $\aleph_2$-cc, then it is trival that $P_{\beta+1}$
does as well. Similarly $P_{\beta+2}$ satisfies
the $\aleph_2$-cc  since $P_{\beta+1}\star Q_{\beta+1}'$
clearly does, and this poset is dense in $P_{\beta+2}$.
The argument for limit ordinals $\beta$ with cofinality less than
 $\omega_2$ is straightforward, so we
 assume that $\beta$ is a limit with cofinality greater
 than $\omega_1$.
  Let $\{ p_\gamma : \gamma\in \omega_2\}$ be a subset
 of $P'_{\beta }$.   Choose any elementary submodel
  $M$ of $H(\kappa^+)$ such that
   $\{p_\gamma : \gamma\in \omega_2\}\in
   M$, $|M|=\aleph_1$, and $M^\omega\subset M$.
Let $M\cap \omega_2 = \lambda$ and 
let $I= \dom(p_\lambda)\cap M$
and fix any $\mu\in M\cap \beta$
so that $I\subset \mu$. For each $\beta\in\mathbf{E}$
such that $\beta+1\in I$, let $\dot a_\beta \in M$ 
so that $\pi_0(p_\lambda(\beta+1)) = \dot a_\beta$.
That is, $p_\lambda(\beta) = (\dot a_\beta, \dot D_\beta)$
for some $\dot D_\beta\in \mathscr D_\beta$.
Clearly the countable sequence $\{ \dot a_\beta :  
 \beta \in I\cap \mathbf{E}\}$ is an element of $M$.
 Therefore there is a $\gamma\in M$ so
 that $\dom(p_\gamma)\cap \mu = I$
 and so that $\pi_0(p_\gamma(\beta+1)) = \dot a_\beta$
 for all $\beta\in \mathbf{E}$ such that $\beta+1\in I$.
 It follows that $p_\gamma\not\perp p_\lambda$.
\end{proof}

 Now we discuss the Cohen real trick, which, though
 simple and powerful, is burdened with 
 cumbersome  notation.

\begin{lemma} Let $\alpha\in \mathbf{E}$
and\label{Cohengeneric}
 let $ p_0\in P_{\alpha+2}\in M$ be a countable
elementary submodel of $H(\kappa^+)$ and let $\delta=M\cap \omega_1$. 
There is a $(P_{\alpha+2}, M)$-generic condition
 $  p_1<p_0$   satisfying that for all 
 $P_{\alpha}$-generic   filters satisfying $p_1\restriction \alpha\in G_0$
 and $\dot Q_{\alpha}$-generic filters 
 $p_1(\alpha)\in G_1$,
   the collection, in $V[G_0\star G_1]$,
    $$ p_1^\uparrow{}_\alpha = \{  p(\alpha+1) : 
  p\in M\cap P_{\alpha+2}, \  p\restriction(\alpha+1)
   \in G_0\star  G_1, \ p_1 < p \}$$
   is $
    \val_{G_0\star G_1} (\dot Q_{\alpha+1}
    \cap M)$-generic
   over $V[G_0\star (G_1\restriction\delta)]$.  
   
   Moreover, for any   $P_\alpha$-name   $\dot Q$  of a ccc
   poset and $P_\alpha\star \dot Q$-generic
  filter   $G_0\star G_2$,
     $p_1^\uparrow{}_\alpha$ is also generic
   over the model $V[G_0\star G_2][  G_1\restriction\delta]$.
\end{lemma}

\begin{proof} 
Let $\dot Q$ be any $P_\alpha$-name of a ccc poset. 
Choose any  $\bar p_1< p_0\restriction(\alpha+1)$ 
that is 
  $(M, P_{\alpha})$-generic with $\bar p_1(\alpha)
   = p(\alpha)$. 
   We will let $p_1\restriction\alpha  
   =\bar p_1\restriction \alpha$ and
   then  we simply have to  choose a value for
    $p_1(\alpha+1)$.
    We may assume that $\bar p_1\restriction
    \mathbf{E} = p_0\restriction\mathbf{E}$.
    Let $\tilde G$ denote the filter $(G_0\star G_1)\cap 
    P_{\alpha+1}(I_\alpha\cap \mathbf{E})$
    and let   $R= (M\cap  \dot Q_{\alpha+1})/\tilde G$. 
       For $r\in R$
      we  may regard $r$
       in the extension $V[\tilde G]$ to have
      the form $(a_r, \dot A_r)$,
      with $a_r\subset\omega_1$, because,
       for each $(\dot a,\dot A)\in M\cap \dot Q_{\alpha+1}$,
       $\dot a$ has support contained
       in $P_{\alpha+1}(I_\alpha\cap \mathbf{E})$.
       We have no such reduction for $\dot A$.
  We adopt the subordering, $<_R$, on $R$ where
 $(a,\dot A)<_R (b,\dot B)$ 
 in $R$  will mean that
  $\mathbf{1}_{P_{\alpha+1}}\Vdash \dot A\subset \dot B$. 
  The fact that $(a,\dot A)\in R$ already means
  that $\mathbf{1}_{P_{\alpha+1}}\Vdash a\subset \dot A$.
  If  $p\in M\cap P_{\alpha+1}$ and
   $(a,\dot A_1)\in R$ is such that
    $p\Vdash (a,\dot A_1) < (b, \dot B)$, then
    there is an $(a,\dot A)\in R$ such that
     $p\Vdash \dot A = \dot A_1$ and
      $(a,\dot A)<_R (b,\dot B)$.

      The quotient poset $(R/\tilde{G},<_R)$ is isomorphic
      to $\mathcal C_{\omega}$. Let $\psi \in V[\tilde G]$
      be an isomorphism from 
      $\mathcal C_{ (\delta,\delta+\omega)}$
      to $(R/\tilde{G},<_R)$.  We regard
      $\mathcal C_{ (\delta,\delta+\omega)}$ as the
      canonical subposet of $\dot Q_{\alpha}$ and
      let $G^\delta_\alpha$ denote a generic filter for
      this subposet of $\dot Q_\alpha$. Now
      we have, in the extension
       $V[\tilde G][G^\delta_\alpha]$,
       a $<_R$-filter $R^\delta_\alpha \subset R$
      given by $\{ \psi(\sigma) : \sigma\in G^\delta_\alpha\}$. 
  Let $a_\omega =\{\delta\}\cup \bigcup
  \{ a_r : r\in R^\delta_\alpha\}$. Note that $\bar p_1$
  forces that $\delta\in \dot C$ for all $\dot C\in 
   M\cap \mathcal D_{\alpha}$. By the construction,
    it follows that we may fix a $P_{\alpha+1}$-name,
     $\dot a_\omega$, for $a_\omega$, that has support
     contained in $I_\alpha\cap \mathbf{E}$. 
     Let $\dot A_\omega$ be the
      $P_{\alpha+1}$-name satisfying that
      $\bar p_1$ forces that
       $\dot A_\omega$ equals the intersection
       of all $\dot C \in \mathcal D_\alpha\cap M$ 
       such that $\dot a_\omega\subset \dot C$.
      It follows that for $r\in R^\delta_\alpha$
      and  $\tilde p \restriction \alpha+1<\bar p_1$,
      $\tilde p(\alpha) \in G^\delta_\alpha$,
      and $\tilde p(\alpha+1) =r$, we have that
      $
      \tilde p \wedge r\Vdash \dot A_\omega\subset
       \dot A_r$  
       (and this takes place in $V[\tilde G]$).
We may choose $\dot A_\omega$ so that 
$\tilde p\Vdash \dot A_\omega = \omega_1$ for all 
$\tilde p\perp \bar p_1$ in $P_{\alpha+1}$. 
It then follows that $(\dot a_\omega,\dot A_\omega)$
is an element of $\dot Q_{\alpha+1}$. We now
define $p_1$ so that $p_1\restriction \alpha+1 = \bar p_1$
and $p_1(\alpha+1) = (\dot a_\omega , \dot A_\omega)$. 
The fact that $p_1$ is $(M,P_{\alpha+2})$-generic follows
from the stronger claim below.

\begin{claim} Let $G_0$ be a $P_{\alpha}$-generic
with $
\bar p_1
\restriction\alpha\in G_0$
and let
 $G_1$ be a 
$\dot Q_\alpha $-generic filter
with $\bar p_1(\alpha)\in M\cap G_1$.
Also let $G_0\star G_2$ be $P_\alpha\star \dot Q$-generic.
Let $\sigma\in \mathcal C_{(\delta,\delta+\omega)} $ be arbitrary.
 Let $\dot D$ be a 
 $P_{\alpha+1}\star \dot Q$-name of a dense 
 subset of 
 $
    \val_{G_0\star G_1} (\dot Q_{\alpha+1}    \cap M)$.
    Then there is a $\tau\supset \sigma$ 
    such that $\tau\Vdash p_1^\uparrow{}_\alpha\cap 
    \val_{G_0\star (G_1 \times  G_2)}(\dot D)\neq\emptyset$.
    \end{claim}

 \bgroup
 
 \def\proofname{Proof of Claim}

\begin{proof}
 Fix the generic filter  $\tilde G\subset G_0\star G_1$
 as used in the construction of $(\dot a_\omega, \dot A_\omega)$
 and let $\psi : \mathcal C^\delta_\alpha
  \rightarrow (R/\tilde G, <_R)$ denote the above mentioned 
 isomorphism. 
    Let $(b , \dot B) = \psi(\sigma)$
    and, using 
    the density of 
      $\val_{G_0\star (G_1\times G_2)}(\dot D)$,
    choose  $(a,A) < (b, \val_{G_0\star G_1}(\dot B))$,
so that    $(a, A)\in \val_{G_0\star( G_1\times G_2)}(\dot D)$.
By elementarity, choose  $(\dot a,\dot A)\in M
\cap \dot Q_{\alpha+1}$ such that 
  $    \val_{G_0\star G_1}( (\dot a, \dot A)) = (a,A)$.
  Again by elementarity and using that
   $\bar p_1$ is $(M,P_{\alpha+1})$-generic,
    there is a $p\in M\cap (G_0\star G_1)$ such
    that $p\Vdash \dot A\subset \dot B$. 
 Now choose $\tau\supset \sigma$ so that
  $\psi(\tau) = (a, \dot A_1)$ satisfies that
   $(a,\dot A_1) <_R (b,\dot B)$ and $p\Vdash
    \dot A_1 = \dot A$. It follows that  
     $\tau\Vdash (a,\dot A_1) \in 
     \val_{G_0\star (G_1\times G_2)}(\dot D)$. 
     Since $p_1\wedge \tau$ also forces that
      $p_1(\alpha+1) < (a,\dot A_1) $ we
      have that $p_1\wedge \tau\Vdash (a,\dot A_1)
       \in p_1^\uparrow{}_{\alpha}$.  
\end{proof}
 
 \egroup

 This completes the proof of the Lemma.
 \end{proof}

\begin{lemma}
 Let $\lambda < \kappa$ with $\lambda\in\mathbf{E}$ 
 and\label{cccpreserve}
  let $\dot Q$ be a $P_\lambda$-name of a ccc poset.
  Then $P_\kappa$  forces that $\dot Q$ is ccc.  
\end{lemma}

\begin{proof}
Let $G$ be a $P_\lambda$-generic filter and
let $Q = \val_G(\dot Q)$.  Since $P_\kappa$ satisfies
the $\aleph_2$-cc, we can assume
that $Q$ is of the form $(\omega_1, <_Q)$. 
We work in the extension 
 $V[G]$ and we view, for each $\lambda< \alpha 
 \leq \kappa$, 
  $\bar P_\alpha=P_\alpha/G$ as a subset
 of $P_\alpha$.
We prove, by induction
on $\lambda\leq\alpha\in \mathbf{E}$, that 
for any countable elementary submodel $\{Q,\lambda,
 \bar P_\alpha\}\in M$ and any $p\in \bar P_\alpha\cap M$,
  there is a $p_M<_E p$ such that
   $(1_Q,p_M) $ is $(M,Q\times \bar P_\alpha)$-generic.
   Note that this inductive hypothesis, i.e. the fact that
   it is $(1_Q,p_M)$ that is  the generic  condition
   rather than $(q,p_M)$ for some other $q\in Q$,
     is equivalent to the statement 
   that $P_\alpha$ preserves that $Q$ is ccc.
   
   The proof at limit steps follows the standard proof
    (as in \cite{properbook})
   that the countable support iteration of proper posets
   is proper.  We feel that this can be skipped. 
   So let $\alpha = \beta+2$ for some $\beta\in \mathbf{E}$. 
   Let $M$ be a suitable countable elementary submodel
   and let $p\in P_\alpha\cap M$ (such
   that $p\restriction \lambda\in G$). 
   Let $M\cap \omega_1=\delta$.
   By the inductive hypothesis, we can assume that we
    have $\bar p_1\in P_\beta$ so that,
     $\bar p_1\restriction\lambda\in G$, 
     $\bar p_1<_E p\restriction\beta$ and so that 
     $(1_Q,\bar p_1)$
     is an $(M,Q\times  P_\beta)$-generic condition.
     Of course it is also clear that $(1_Q,\bar p_1)$
     is an $(M, Q\times P_{\beta+1})$-generic 
     condition.  Now let
   $ p_1\in P_{\beta+2}$ be chosen as in Lemma 
 \ref{Cohengeneric}. That is,  $p_1$ is chosen
 so that for any $P_{\beta}$-generic filter $G_\beta\supset
 G$ with $p_1\restriction\beta\in G_\beta$,
 any $\mathcal C_{\omega_1}$-generic $G_1$
 with $p_1(\beta)\in G_1$, and,
 since  $  Q $ is ccc in $V[G_\beta]$,
  any $Q$-generic filter
      $G_Q$,  we have that
  $p_1^\uparrow{}_\beta$ is generic
 over $V[G_\beta\star (G_1\times G_Q)]$.
 Let $G_{\beta+1} = G_\beta\star G_1$.

  Let $D\in M$ be any dense open subset
  of $P_{\beta+2}\star Q$. Let $R$ denote
  $\dot Q_{\beta+1}/(G_{\beta}\star G_1)$.
  It follows that
   $ D/(G_{\beta}\star G_1)  $ or 
   $$E
    = \{ (r,q  ) : (\exists d\in D)\ 
(  d{\restriction}\beta{+}1\in G_{\beta}\star G_1 \ 
\&\
  d = d{\restriction}\beta{+}1\star (r,  q))\}$$
  is a dense open subset of 
  $R\times Q$
  and $E\in M[G_{\beta+1}]$. By standard product
  forcing theory,  we have that
 for each $r\in R$, $E_r  = \{ q\in Q: 
  (\exists s\in R)(s<r \ \&\ (s,q )\in E\})$ is a dense
  subset of $Q$. For each $r\in R \cap M[G_{\beta+1}]$,
   $E_r\in M[G_{\beta+1}]$ and so, $E_r\cap M[G_{\beta+1}]$
   is a predense subset of $Q$.  
  This implies that, for each $\bar q\in Q$,
   the set $E(\bar q) = \{ s\in R\cap M[G_{\beta+1}] : 
    (\exists (s,q)\in E\cap M[G_{\beta+1}])( \bar q\not\perp q)\}$
    is a dense subset of $R\cap M[G_{\beta+1}]$. 
    Although $E(\bar q)$ need not be an element of
    $M[G_{\beta+1}]$, it is an element of 
     $V[G_{\beta} \star (G_1\restriction\delta) ]$.   Therefore, 
      by Lemma \ref{Cohengeneric}, 
       $E(\bar q)\cap p_1^\uparrow{}_\beta$
       is not empty   
       for all $\bar q\in G_Q$. By elementarity,
        it then follows that $p_1$ is an
         $(M,P_{\beta+2}\star Q)$-generic condition.
  \end{proof}

\section{S-space tasks}

Following \cite{ARS} and \cite{stevoSspace} we define
a poset of finite subsets of $\omega_1$ separated by a cub.

\begin{definition}
 For a family $
 \mathcal U = \{ U_\xi : \xi\in\omega_1\}$ and a cub $C\subset \omega_1$,
  define the poset $Q(\mathcal U,C) \subset [\omega_1]^{<\aleph_0}$, 
  to be the set of finite sets $H\subset\omega_1$ such that
  for $\xi<\eta$ both in $H$
\begin{enumerate}
 \item   $\xi\notin U_\eta$  and $\eta\notin U_\xi$, 
  \item there is a $\gamma\in C$ such that $\xi <\gamma\leq\eta$.
\end{enumerate}
$Q(\mathcal U, C)$ is ordered by $\supset$. 
\end{definition}

\begin{definition}
 A family $\mathcal U=\{U_\xi : \xi <\omega_1\}$ is an S-space
 task if it satisfies: 
\begin{enumerate}
 \item $\xi\in U_\xi \in [\omega_1]^{<\aleph_1}$,
 \item every uncountable $A\subset\omega_1$ has a countable
 subset that is not contained in any
  finite union from the family $\mathcal U$. 
\end{enumerate}
\end{definition}

\begin{remark} If  $\mathcal T$ is a regular locally countable topology on
 $\omega_1$ that contains no uncountable free sequence
 (see Definition \ref{freesequence}),
 then each neighborhood assignment $\{ U_\xi : \xi\in\omega_1\}$
 consisting of open sets with   countable closures, 
 is an S-space task. An uncountable $A\subset\omega_1$ failing
 property (2) would contain an uncountable free sequence.
  Suppose
 that there is a 
  cub $C\subset\omega_1$ such that
  $Q(\mathcal U,C)$ is ccc.  
  Then, as usual, there is a 
   $q\in Q(\mathcal U,C)$ such that any generic filter
   including $q$ is uncountable. 
    If $G\subset Q(\mathcal U,C)$
   is a filter (even pairwise compatible),
    then $\bigcup G$ is a discrete subspace of 
   $(\omega_1,\mathcal T)$. 
   Of course this cub $C$ can be assumed to satisfy
  that if $\xi<\eta$ are separated by $C$, then $\eta\notin
   U_\xi$. This means that requirement (1) in
   the 
   definition of $Q(\mathcal U,C)$
   can be weakened to only 
    require that $\xi\notin U_\eta$.
    \end{remark}

 The following result is a restatement of
 Lemma 1 
 from
  \cite{stevoSspace}.   It also uses
  the Cohen real trick.  
  We present a proof
  that is more adaptable to the modifications
  needed for the consistency with $\mathfrak c > \aleph_2$.

\begin{proposition}
Let $R$ be a ccc poset
and let 
  $\mathcal U=\{\dot U_\xi : 
  \xi\in\omega_1\}$ be a sequence of
  $R$-names such that $\mathcal U$ is
  forced to be   an S-space task\label{step1}.
Then 
 $R\times P_{2 }  $ forces that 
 for every $n\in\omega$, every 
 uncountable pairwise disjoint subfamily  
 $ \mathcal H$ of $
 Q(\mathcal U,\dot C_1)\cap [\omega_1]^n$,
  has
 a countable subset $\mathcal H_0$ 
 satisfying that, for some $\delta\in\omega_1$ and 
 all $F\in [\omega_1\setminus \delta]^n$,
 there is an $H\in\mathcal H_0$ such
 that $H\cap \bigcup\{ U_\xi : \xi \in F\} =\emptyset$.
  In
 particular, $R\times P_2  $ forces
 that $Q(\mathcal U,\dot C_1)$ is ccc.
\end{proposition}
 
\begin{proof}
 Of course $P_{2 }$ is isomorphic to 
  $\mathcal C_{\omega_1}
  \star \dot{\mathscr J}$.
  Fix any $n\in\omega$ and let
  $\{ \dot H_\xi : \xi \in \omega_1\}$ 
  be $R\times P_{2 } $-names of 
  pairwise disjoint   elements of
   $[\omega_1]^n\cap Q(\mathcal U, \dot C_1)$.
     Since we can
    pass to an uncountable subcollection 
    of $\{ \dot H_\xi : \xi\in\omega_1\}$ 
  we may  assume   that for all $\xi\in\omega_1$, 
   it is forced that 
    there is a $\delta\in \dot C_1$ such that
     $\xi < \delta\leq \min(\dot H_\xi)$.

    For each
   $(r,p)\in R\times P_2$
   and $H\in [\omega_1]^n$,
    let $\Gamma_\xi(H,(r,p))$ 
     be the set   $ \{  s  \in R : (\exists q\in P_2)(
        (s,q)<(r,p)\ \&  \ (s,q)  \Vdash H=\dot H_\xi)\}$.
     In other words, $\Gamma_\xi(H,(r,p))$ is not empty
     if and only if $(r,p)\not\Vdash H\neq \dot H_\xi$.
We say that $\Gamma_\xi(H,(r,p))$ is 
 $\omega_1$-full simply if it is not empty.
 \medskip

   Now we define what it means for 
   $\Gamma_\xi(H,(r,p))$ to be $\omega_1$-full for
     $H\in [\omega_1]^{n-1}$. We require that
     there is a set
   $\{ \dot \eta_\zeta : \zeta\in\omega_1\}$ of
    canonical $R$-names such that
   $r\Vdash \dot \eta_\zeta \in  \omega_1\setminus
    \zeta$
   and for
  $(\eta,s)\in \dot 
  \eta_\zeta$,  $s\leq r$  and satisfies
  that   $   \Gamma_\xi(H\cup \{  \eta \} , (s,p))$ is 
  $\omega_1$-full.  It is worth noting
  that $(r,p)$ has been changed to 
  $(s,p)$ rather than to some $(s,q)$ with $q<p$.
  This definition generalizes to  $H\in [\omega_1]^{i}$.
  We say that $\Gamma_\xi(H,(r,p))$ is $\omega_1$-full 
      if there is a  set of canonical  $R$-names $\{ \dot \eta_\zeta :
       \zeta\in\omega_1\}$ 
       such that, for each $\zeta\in\omega_1$,
        $r\Vdash \dot \eta_\zeta \in 
         (\omega_1\setminus\zeta) $,
         and for $(\eta , s)\in \dot \eta_\zeta$,
         $s\leq  r $ and
          $\Gamma_\xi(H\cup \{\eta\} , (s,p))$
          is $\omega_1$-full.

\bgroup
\def\proofname{Proof of Claim:\/}

\begin{claim}
Suppose that $\Gamma_\xi(\emptyset, (r,p))$ is $\omega_1$-full
 and   that\label{two} 
   $M\prec H(\kappa^+)$ is countable
   and $\{\xi, \mathcal U,R,   (r,p) \}\in M$.
Then for any $\bar r<r \in R$ and  
finite  
$F\subset \omega_1\setminus M$,
  there are $(s,q), H\in M$
   such that 
   
\begin{enumerate}
 \item $(s,q) <(r,p) \in R\times P_2$,  
 \item $H\cap \bigcup\{ \dot U_\zeta : \zeta \in F\}$ is empty,
 \item $(s,q)\Vdash  \dot H_\xi = H$,
 \item $s\not\perp \bar r$.
\end{enumerate}
\end{claim}

\begin{proof} 
Let $\dot W_F = \bigcup\{ \dot U_\zeta : \zeta \in F\}$.  
 Since $R\in M\prec H(\kappa^+)$  is
 ccc and forces that $ 
 \mathcal U$ is an
  S-space task, it follows that for each 
  $R$-name 
   $\dot A\in  M$ for an uncountable subset
   of $\omega_1$, 
    the set $\dot A\cap M$ is forced to not
    be contained in $\dot W_F$. 
    By induction on $1\leq i\leq n$, we choose
     $(\eta_i, s_i)\in (\omega_1\times R)\cap M$
     and $\bar r_i<s_i$ 
     so that $\bar r_i \Vdash \eta_i\notin \dot W_F$,
      $s_i  \leq  s_j \leq  r$ 
      and
      $\bar r_i \leq \bar r_j$ for $j<i$,
     and $ \Gamma_\xi(\{\eta_j : 1\leq j <  i\} , (s_i,p))$ is
     $\omega_1$-full. 
 
 Let $\bar r_0 =\bar r$, 
$(s_0,q_0) = (r,p)$, 
$\emptyset = \{ \eta_j : 1 \leq j  < 1\}$ and we 
assume by induction
that, at stage $i$, $\Gamma(\{ \eta_j :1\leq  j<i\}, (s_{i},p))$
is $\omega_1$-full. Fix any
 sequence $\{ \dot \eta_\zeta : \omega\leq \zeta\in \omega_1\}\in M$
     witnessing that $\Gamma_\xi(
     \{\eta_j : j<i \},(s_{i},p ))$ is $\omega_1$-full.
      We have that 
 $\{ \dot \eta_\zeta : \omega\leq \zeta\in\omega_1\}\in M$ is
 an $R$-name  for an uncountable subset of $\omega_1$. 
  It follows that  $\bar r_{i-1}$ forces that there is a
  $\zeta\in M$ such that $\dot \eta_\zeta \notin \dot W_F$.
We find an extension $\bar r_{i+1}$ of $\bar r_{i}$
so that we may choose $\zeta\in M$
   and  
   $(\eta, s)\in \dot \eta_\zeta$ such that $\eta\notin \dot W_F$,
    $\bar r_{i+1} < s\leq s_i$. 
   Therefore we set $(\xi_i, s_{i+1}, q_{i+1} ) = 
    (\eta, s, q)$ and this completes the construction.
    
   Setting $H=\{\xi_i :1\leq  i \leq  n\}$ and $(s,q) =
    (s_n,q_n)$  completes the proof of the Claim.     
\end{proof}
 
   \begin{claim}
 If $\Gamma_\xi( H,(r,p))$ is not\label{three}
 $\omega_1$-full, there is an $s<r$  in $R$  and a 
 $\zeta<\omega_1$ 
 such that  $\Gamma_\xi(H\cup \{\eta\}, (s,p))$
 is not $\omega_1$-full  for all $\zeta < \eta\in  \omega_1$. 
\end{claim}
  
\begin{proof} 
Since $\Gamma_\xi(H, (r,p))$ is not
$\omega_1$-full, there is some
 $\zeta\in\omega_1$ so that 
 the suitable nice name $\dot \eta_\zeta$
 does not exist. It follows immediately
 that $\dot \eta_{\gamma}$ does not
 exist for all $\zeta < \gamma\in\omega_1$. 
 In addition,  since $\dot \eta_\zeta$ fails
 to exist, it is because 
  $\Gamma_\xi(H\cup \{\eta\}, (s',r))$ is
  not $\omega_1$-full for all $s'\not\perp s$.
 \end{proof}

\begin{claim}
 For every $(r,p)\in R\times P_2$, there\label{four} is a $\delta$ so
 that $\Gamma_\delta(\emptyset, (r,p))$ is $\omega_1$-full.
\end{claim}

\begin{proof}
 Let $M_0$ be a countable elementary submodel of $H(\kappa^+)$
 so that $\{ \mathcal U, (r,p), R\}\in M_0$. 
 Choose any $p_1 <_E p$ (i.e.
  $p_1(0)=p(0)$ and
   $p_1(0)\Vdash p_1(1)<p(1)$)
    that is $(M_0,P_2)$-generic. Notice
 that $(r,p_1)$ is therefore $(M, R\times P_2)$-generic since
  $R$ is ccc.
 Let $\delta_0 = M_0\cap \omega_1$. 
 Choose any continuous $\in$-chain 
  $\{ M_\alpha : 0<\alpha < \omega_1\}$ of countable
  elementary submodels of $H(\kappa^+)$ such 
  that $p_1\in M_1$. 
  For each $\alpha\in\omega_1$, let $\delta_\alpha = 
   M_\alpha\cap \omega_1$.  We did not actually
   have to choose $p_1$ before choosing $M_1$ of
   course. Let $C $ be the cub
    $\{ \delta_\alpha :\alpha\in\omega_1\}$
     and let $p_2 \in P_2 $ be a common
     extension of $p_1$ and $
     (\emptyset, (\emptyset, \delta_0\cup 
     (C\setminus\delta_0)))$
     (or equivalently $ p_2(0)\leq p_1(0)$
     and $p_2(0)\Vdash p_2(1) \leq
     (\pi_0(p_1(1)), \pi_1(p_1(1))\cap C)$).
     It follows that $p_2\Vdash \dot C_1\setminus \delta_0\subset C$.

 Assume $\Gamma_{\delta_0}(\emptyset, (r,p))$
 is not $\omega_1$-full. Choose 
  $s_0<r$ and $\zeta_0\in\omega_1$ as in
  Claim \ref{two}.   
 By elementarity we may assume that $s_0,\zeta_0$
 are in $  M_1$.

  Now choose any $\bar s_0<s_0 $
  so  that there is a $q_0<p_1$ and
  an $H\in [\omega_1\setminus\delta_0]^n$ 
 such that $(\bar s_0,q_0)\Vdash \dot H_{\delta_0}= H$.
   Of course this implies that 
$ \Gamma_{\delta_0}(H,(r,p))$ is not empty
and therefore, it is $\omega_1$-full.
  Let $H$ be enumerated in increasing order $\{
  \eta_i :1\leq i\leq n\}$.
 
Since $(\bar s_0,q)\Vdash \dot H_{\delta_0}\in 
 Q(\mathcal U, \dot C_1)$, we can assume that
  $q$ has already determined the members of $\dot C_1$
  that separate the elements of $\{\delta_0\}
  \cup H$.  In other words,
     there is
   a set $\{ \alpha_i : 1\leq i \leq n\}\subset \omega_1$ so
   that  $\{ \delta_{\alpha_i} : 1\leq i 
   \leq n \}\subset 
    \pi_0(q(1))\subset C$ 
    such that, for each $1\leq i<n$,
      $\delta_0\leq \delta_{\alpha_{i-1}} \leq \eta_i  $.
 Therefore, $\{ \eta_j : 1\leq j<i\} \in M_{\alpha_i}$ for all $i<n$
 and $\Gamma_{\delta_0}(\{ \eta_j : 1\leq j\leq n\},(r,p)) $ is
  $\omega_1$-full. Clearly, for all $s'<\bar s_0$,
    $\Gamma_{\delta_0}(\{\eta_j : 1\leq j\leq n\}, (s',p))$ 
    is also $\omega_1$-full.
  
  By the choice of $s_0$ and $\zeta_0$,
   we have that $\Gamma_{\delta_0}(\{\eta_1\},
      (s_0, p))\in M_{\alpha_2}$ is not $\omega_1$-full.
We note that 
      $\bar s_0$ is $(M_{\alpha_2},R)$-generic
      condition.
      There is therefore, by Claim \ref{two},
       a $\zeta_1\in M_{  \alpha_2}$ and a pair
       $\bar s_1 < s_1$ so that $s_1\in M_{\alpha_2}$,  
        $\bar s_1<\bar s_0$ and $\Gamma_{\delta_0} 
        (\{\eta_1,\eta\} , (s_1, p))$ is not $\omega_1$-full
        for all $\eta > \zeta_1$. Following this procedure
        we can recursively choose
        a pair of
        descending sequences $\{ s_i : 1\leq i\leq n\}\subset R$
        and $\{ \bar s_i : 1\leq i\leq n\}\subset R$ so that
        
\begin{enumerate}
 \item $s_{i-1}  \in M_{\alpha_i}$ and  $\bar s_i < s_i$, 
 \item $\Gamma_{\delta_0}(\{\eta_1,\ldots, \eta_i\}, (s_i,p))$ is not
  $\omega_1$-full.
\end{enumerate}
  We now have a   contradiction 
  that completes the proof. We noted above that
  since $\bar s_n < \bar s_0$, 
   $\Gamma_{\delta_0}(
   \{\eta_1,\ldots, \eta_i\}, (\bar s_n ,p))$ is
  $\omega_1$-full.
  However since
    $\bar s_n<s_n$, this contradicts that 
     $\Gamma_{\delta_0}(
   \{\eta_1,\ldots, \eta_n\}, (s_n,p))$ is not $\omega_1$-full. 
\end{proof}
\egroup

Now we complete the proof of the Proposition. 
Consider any countable elementary submodel $M$
as in Claim \ref{two}  and
let $\delta=M\cap \omega_1$.
Let $p_1$ be a condition as in Lemma \ref{Cohengeneric}
applied to the case $\alpha=0$. 
 Let $  G_R$ be any $R$-generic
filter and 
let $G_1\subset\mathcal C_{\omega_1}$
 be any
 generic filter,   which
 is generic over the model $V[G_R]$.
 Pass to the extension $V[G_R]$.

  Fix any $F\in [\omega_1\setminus\delta]^n$.
   It follows from Claim \ref{two}   and Claim \ref{three}, 
   that the set 
$  \mathcal W_F $ of those $
     (t,(\dot b, \dot B)) \in  
     M\cap (\mathcal C_{\omega_1} \star \dot{\mathscr J})  
      $
 for which
 \begin{multline*}
 (\exists \xi\in \delta) (\exists s\in G_R) \ \ 
   ( s\Vdash  H\cap  \dot W_F  =\emptyset\  \&\ 
         (s,(t,(\dot b,\dot B) ))\Vdash H = \dot H_\xi )
       \end{multline*}
       is a dense subset of 
       $M\cap (\mathcal C_{\omega_1}\star \dot {\mathscr J})$.
        The proof is that Claim \ref{three}  provides 
  a  potential $\xi \in M$ to strive for, and Claim \ref{two} provides
  an $(s,q)$ to yield an element of $ {\mathcal W}_F$.
  
  It then follows easily that, in the extension
   $V[G_R\times G_1]$, the  
   set 
   $$
   \val_{  G_1\restriction\delta}(\mathcal W_F)
   = \{ \val_{G_1}(\,(\dot  b, \dot B)\,) : (\exists t\in G_1)\ 
   ( (t, (\dot b, \dot B)))\in \mathcal W_F\}$$
   is a dense subset of $\val_{G_1}(M\cap \dot {\mathscr J})$
   which is an element of $V[G_R\times (G_1\restriction \delta)]$. 
   Since $p_1$ forces that the generic
   filter meets $\val_{G_1\restriction\delta}(\mathcal W_F)$,
    this completes the proof.
\end{proof}

 For any $\alpha\leq \kappa$ and subset $I\subset\alpha$,
  we will say that a $P_\alpha$-name $\dot E$
  is a $P_\alpha(I)$-name
  if it is a $P_\alpha(I)$-name in the usual recursive sense.
  This definition makes technical sense even if
   $P_\alpha(I)$ is not a complete subposet of $P_\alpha$. 
  
\begin{corollary}
 Let $\lambda\in \mathbf{E}$ and let 
 $\dot R_0$ be a $P_{\lambda}(I_\lambda) $-name 
 that is forced by $P_\lambda$ to be
  ccc poset. Let $\dot R$ be a $P_\lambda$-name of a
   ccc poset such $\mathbf{1}_{P_\lambda}$ forces 
    that $\dot R_0\subset_c \dot R$.
 Assume that $\mathcal U = \{\dot U_\xi : \xi\in\omega_1\}$
 is a sequence of $P_\lambda(I_\lambda) \star \dot R_0$-names 
 of subsets
 of $\omega_1$ such\label{step2}
 that $P_{\lambda}\star \dot R$ forces 
 that $\mathcal U$ is an S-space task.  
 Then the $P_{\lambda+2}$-name $Q(\mathcal U,\dot C_\lambda)$
 satisfies that $P_{\lambda+2}$ forces that $  \dot R\times 
 Q(\mathcal U,\dot C_\lambda)$ is ccc.
\end{corollary}

\begin{proof}
Let $G_\lambda$ be a $P_\lambda$-generic filter
and pass to the extension $V[G_\lambda]$. 
Let $R = \val_{G_\lambda}(\dot R)$
and observe that we may now
regard $\mathcal U$ as a family of $R$-names
of subsets of $\omega_1$ that is forced to be an
S-space task.  We would like to simply apply
Lemma \ref{step1} but 
unfortunately,
 $P_{\lambda+2}$ is not isomorphic to
  $P_\lambda\star P_2$. 
Naturally the difference is that
 $\dot Q_{\lambda+1}$ is a proper subset
 of $\dot {\mathscr J}$. It will suffice to identify
 the three key places in the proof of Lemma \ref{step1}
 that depended on consequences of
 the properties of $\mathscr J$
 and to verify that the consequences also
 hold for $\dot Q_{\lambda+1}$. The first
 was in the proof of Claim \ref{four}  where we selected
 a condition $p_2(1)\in \mathscr J$ that satisfied
 that $\pi_1(p_2(1))$ was forced to be a subset
 of  $C\cup \delta_0$ for the cub $C$. Since, in this proof,
 $C$ will be an cub set in the model   $V[G_\lambda]$, 
 it follows from condition (6) of Definition \ref{conditions},
 this can be done. The next property of $P_2$
 that we used was that Lemma \ref{Cohengeneric} holds, 
 but of course this also holds for $P_{\lambda+2}$. 
 The third is in the proof and statement of Claim \ref{two}. 
 When choosing the pair $(s,q)$ in $R\times P_2$
we require that it satisfies condition (2) in Claim \ref{two}. 
In the current situation, each $\dot U_\zeta$ is not
simply an $R$-name but rather it is a
 $P_\lambda(I_\lambda)\star \dot R_0$-name. 
Therefore, there is a $P_\lambda(I_\lambda)$-name
 for a suitable $q$ so that $(s,q)\Vdash H\cap \bigcup
 \{\dot U_\zeta : \zeta\in F\}$ is empty. This causes
 no difficulty since $P_\lambda(I_\lambda)$-names
 for elements of $\dot Q_{\lambda+1}$ are, in fact,
 elements of $\dot Q_{\lambda+1}$.
  That is, a choice for $(s,q)$ in $R\times 
   (\dot Q_{\lambda}\star \dot Q_{\lambda+1})$ can
   be made in $V[G_\lambda]$ as required in Claim \ref{two}.
\end{proof}

    \section{Building the final model}
    
In this section we present the construction of
the iteration sequence of length $\kappa+\kappa$
extending that of Definition \ref{conditions} that will
be used to prove the main theorem.

We introduce more terminology.

\begin{definition}
  Fix\label{defineQ} any $\mu\leq \lambda\leq \kappa$
 and define   $\mathcal Q(\lambda,\mu)$ to be the set of
  all iterations 
   $\mathbf{q} $
   of the form $
   \langle P^{\q}_\alpha , \dot Q^{\q }_\beta : 
   \alpha \leq \lambda+\mu, 
    \beta < \lambda+\mu\rangle \in H(\kappa^+)$
    satisfying that

\begin{enumerate}
 \item      $\langle P^{\q}_\alpha , \dot Q^{\q}_\beta : \alpha\leq\lambda, 
       \beta <\lambda\rangle$ is our sequence
        $\langle P_\alpha, \dot Q_\beta : \alpha \leq \lambda,
         \ \beta <\lambda\rangle$ 
       from Section 2,
       \item   for
       all $\lambda\leq \beta < \lambda + \mu$, $\dot Q^{\q}_\beta
       \in H(\kappa)$
       is a $P^{\q}_{\beta}$-name of a ccc poset,
        \item for all $\alpha \leq \mu$ and $p\in P^{\q}_{\alpha}$,
  $p\restriction \lambda\in P^{\q}_\lambda$ and
   $\dom(p)\setminus \lambda$
  is finite,
  \item if $\lambda<\kappa$, then $\q\in H(\kappa)$.
\end{enumerate}
   For $\q\in \mathcal Q(\lambda,\mu)$, 
  let $\q(\kappa)$ denote the element of
  $ \mathcal Q(\kappa,\mu)$ where
   $\dot Q^{\q(\kappa) }_{\kappa+\beta} =
    \dot Q^{\q}_{\lambda+\beta}$ for all $\beta < \mu$.
 \end{definition}

\begin{lemma}
 Let  $\mu<\kappa$ and let
 $\q   \in \mathcal Q(\kappa,\mu)$
    and let  
   $\mathcal U = \{ \dot U_\xi : \xi\in\omega_1\}$
 be a sequence of $P^\q_{\kappa+\mu} $-names. 
 Assume that $P^\q_{\kappa+\mu}$ forces \label{killS}
 that $\mathcal U$ is an S-space task. 
 Let $\bar M$ be an elementary submodel
 of $H(\kappa^+)$  of cardinality $\aleph_1$
 that is closed under $\omega$-sequences and
 contains $\{ \mathcal U,\q\}$. 
 Choose any $ \lambda\in \mathbf{E}\cap\kappa$
 so that $\bar M\cap \kappa\subset I_\lambda$. 
 Then $P^\q_{\kappa+\mu}$ forces
 that $Q(\mathcal U, \dot C_{\lambda})$ is ccc. 
 \end{lemma}

\begin{proof}
Since $\mu\in \bar M$, it follows that 
 $\mu\leq\lambda$.  Furthermore, 
 by the assumptions on $\q\in \mathcal Q$
 and $\q\in \bar M$, 
  it follows that there is a $\gamma\in \bar M\cap \kappa$
  such that 
  $\dot Q_\beta$ is a $P_\gamma$-name
  for all $\kappa\leq \beta < \kappa+\mu$.  
  In addition, for each $\beta\in \bar M\cap \mu$,
   $\dot Q_\beta$ is a $P_\gamma(\bar M\cap \gamma)$-name.
  Since $\gamma<\lambda$, there is a
  $P_\lambda$-name,
  $\dot R $, of a  
  finite support iteration of length $\mu$ such
 that   $P_\kappa \star \dot R$ is isomorphic
   to $P^\q_{\kappa+\mu}$.
   More precisely, the  $\beta$-th iterand for
    $\dot R$ is the name $\dot Q_{\kappa+\beta}$. 
   Similarly, let $\dot R_0$ be the set of 
   conditions in $\dot R$ with support contained
   in $\bar M\cap \mu$ and values taken in $\bar M
   \cap \dot Q_{\kappa+\beta}$ for each $\beta$
   in the support. 
Then  we have that 
$\mathbf{1}_{P_\lambda} \Vdash 
    \dot R_0 \subset_c \dot R$.
By minor re-naming, we may treat
    $\mathcal U$ as a sequence of 
    $P_\lambda(I_\lambda)\star \dot R_0$-names.
   Since $P^\q_{\kappa+\mu}$ forces
   that $\mathcal U$ is an S-space task,
   it follows that $P_\lambda\star \dot R$ 
   also forces that       
   $\mathcal U$ is an S-space task.
   By Corollary \ref{step2}, $P_{\lambda+2}$
   forces that $\dot R\times Q(\mathcal U,\dot C_\lambda)$
   is ccc.   
 By Lemma \ref{cccpreserve},  $P_\kappa$
 forces that $\dot R\times Q(\mathcal U,\dot C_\lambda)$ is
 ccc. Since $P^\q_{\kappa+\mu} $ is isomorphic
 to $P_\kappa\star \dot R$, this completes the proof.
   \end{proof}

\begin{theorem}
Let $\kappa>\aleph_2$ be a regular\label{mainthm} cardinal in 
 a model of GCH. 
 There is an iteration sequence
  $\langle P_\alpha , \dot Q_\beta : \alpha \leq \kappa+\kappa,
   \ \beta < \kappa+\kappa\rangle$ such that 
    $P_{\kappa+\kappa}$ forces that there are no S-spaces
    and, for all $\mu < \kappa$, 
  $\langle P_\alpha , \dot Q_\beta : \alpha \leq \kappa+\mu,
   \ \beta < \kappa+\mu\rangle$  is in $\mathcal Q(\kappa,\mu)$.
   It therefore follows that $P_{\kappa+\kappa}$ is
   cardinal preserving and forces that
    $ \kappa^{<\kappa} = \kappa = \mathfrak c$.   
    
    The iteration can be chosen so that, in addition,
     Martin's Axiom holds in the extension.  
\end{theorem}

\begin{proof}
Fix a sequence $\mathscr I = \{ 
 I_\gamma : \gamma \in \kappa\}$ as described in 
 the construction of the
  sequence $\langle P_\alpha , \dot Q_\beta : 
   \alpha\leq \kappa, \  \beta < \kappa\rangle$. 
  Also let $\mathcal
   Q(\lambda,\mu)$ for $\mu\leq \lambda <\kappa$
  be defined as in 
  Definition \ref{defineQ}.

 We introduce still more notation.    
    For all $\alpha \leq \lambda < \kappa$, 
 let $P^\lambda_\alpha$
    simply denote $P_\alpha$
    and $\dot Q^\lambda_\alpha = \dot Q_\alpha$. 
Also   for any 
  $\mu \leq  \lambda<\kappa$ and sequence 
  $\q' = 
  \langle \dot Q_\beta' : \beta < \mu \rangle\in H(\kappa)$,
 let   $\dot Q_{\lambda+\beta}^\lambda (\q')$ denote $ \dot Q_\beta'$. 
 By recursion on $\alpha <\mu$, let $P^\lambda_{\lambda+\alpha}(\q')$
 denote the limit of the iteration sequence
  $\langle P^\lambda_{\zeta} (\q'), \dot Q^\lambda_{\beta}(\q') : 
     \zeta < \alpha, \ \beta <\alpha\rangle$ so long as
     this sequence
     (and its limit) is in $\mathcal Q(\lambda,\alpha)$. 
  Say that a sequence
   $\q' = \langle \dot Q_\beta' : \beta < \lambda \rangle\in H(\kappa)$
  is suitable  if for all 
  $\alpha \in \mathbf{E}\cap   \lambda{+}1$,
   $\langle P^\lambda_\zeta(\q'), \dot Q^\lambda_{\beta}(\q') :
\zeta  \leq    \alpha  , \ \beta < \alpha\rangle$ is in
 $\mathcal Q(\lambda,\alpha)$. We state for reference 
 two properties of suitable sequences.

\begin{fact}
 If $\lambda$ is a limit ordinal, then 
 $\langle \dot Q_\beta' : \beta\in \lambda\rangle \in H(\kappa)$
 is suitable so long as
  $\langle \dot Q_\beta' : \beta < \mu\rangle$
 is suitable for all $\mu<\lambda$.
\end{fact}

\begin{fact}
 If $\q' = 
 \langle \dot Q_\beta' : \beta \in \lambda\rangle \in H(\kappa)$
 is suitable, then $\langle \dot Q_\beta' : \beta \in \lambda+1\rangle$
 is suitable for any $P^\lambda_{\lambda+\lambda}(\q')$-name
 $\dot Q_\lambda'$ of a ccc poset of cardinality at most $\aleph_1$.
\end{fact}

Now that we have this cumbersome, but necessary, notation
out of the way, the proof of the theorem is a routine
consequence of the prior results. Let $\sqsubset$ be a
well ordering of $H(\kappa)$ in type $\kappa$. 
We recursively define a sequence $\langle
 \dot Q_\beta' : \beta < \kappa\rangle$ and a 
 1-to-1 sequence
  $\langle \mathcal U_\beta : \beta <\kappa\rangle$.
  One
 inductive assumption is that every initial segment
 of  $\langle
 \dot Q_\beta' : \beta < \kappa\rangle$ 
 is a suitable sequence.  The list $\{\mathcal U_\beta :
 \beta < \kappa\}$ will contain the list the potential S-space
 tasks as we deal with them.

 Let
  $\lambda<\kappa$ and assume that 
  $\langle \dot Q_\beta' , \mathcal U_\beta
  : \beta < \lambda\rangle \in
  H(\kappa)$ has been chosen.
  If $\lambda\notin\mathbf{E}$, then 
  $\dot Q_\lambda'$ is the trivial poset
  and $\mathcal U_\lambda=\lambda$.
  Now let $\lambda\in \mathbf{E}$
  and let
   $\q' =
  \langle \dot Q_\beta' : \beta < \lambda\rangle$. 
Consider the set of all  $P^\lambda_{\lambda+\lambda}(\q')$-names
$\mathcal U=\{ \dot U_\xi : \xi\in\omega_1\}$
that are forced to be S-space tasks. Consider
only those $\mathcal U$ for which
  there is an elementary
submodel $\bar M$ of $H(\kappa^+)$ as in Lemma 
 \ref{killS}. More specifically,
 such that   $\bar M\cap \lambda\subset I_\lambda$,
 $\{\mathcal U, P^\lambda_{\lambda+\lambda}
 (\q')\} \in \bar M$, $|\bar M|=\aleph_1$, and
  $\bar M^\omega \subset \bar M$. 
  The final requirement of such $\mathcal U$
  is that they are 
   not in the set $\langle \mathcal U_\beta :\beta <\lambda\rangle$.
  If any such $\mathcal U$ exist, then let
   $\mathcal U_\lambda$ be the $\sqsubset $-minimal one. 
   Loosely, $\mathcal U_\lambda$ is the 
    $\sqsubset$-minimal S-space task that has not yet
    been handled and can be handled at this stage.
    Otherwise, let $\mathcal U_\lambda=\lambda$ 
    (so as to preserve the 1-to-1 property).
    Now we choose $\dot Q_\lambda'$. 
    If $\mathcal U_\lambda =\lambda$, then
     $\dot Q_\lambda$ is the trivial poset. 
     Otherwise, of course, $\dot Q_\lambda$ is
     the $P^{\lambda+2}_{\lambda+\lambda}(\q')$-name
     for
      $Q(\mathcal U_\lambda, \dot C_\lambda)$. 
      By Lemma \ref{killS} and Fact 2,
       $\langle \dot Q_\beta : \beta \leq\lambda\rangle$
       is suitable.
       
       This completes the recursive construction of the
        suitable sequence $\q' =
        \langle \dot Q_\beta' : \beta < \kappa\rangle$
        and the listing $\langle \mathcal U_\beta : \beta < \kappa\rangle$.
        It remains only to prove that if
         $\mathcal U = \{ \dot U_\xi : \xi\in\omega_1\}$
         is a $P^\kappa_{\kappa+\kappa}(\q')$-name of
         an S-space task,  then there is an $\alpha < \kappa$
         such that 
          $\mathcal U = \mathcal U_\alpha$. 
          Fix any such $\mathcal U$ and elementary
          submodel $\bar M\prec H(\kappa^+)$ such
          that $\{\mathcal U, 
          P^\kappa_{\kappa+\kappa}(\q')\}\in\bar M$,
          $|\bar M| = \aleph_1$, and $\bar M^\omega\subset 
          \bar M$. Let $\Lambda$ be the set of $\lambda\in \kappa$
          such that $\bar M\cap \kappa \subset I_\lambda$. 
          Let $\gamma$ be the order type of the set
          of predecessors of $\mathcal U$ in the well ordering
         $\sqsubset$. Choose any $\lambda\in \Lambda$ 
         such that the order type of $\Lambda \cap \lambda$
         is greater than $\gamma$. Note that $\Lambda\subset
          \mathbf{E}$. For every $\mu\in \Lambda \cap\lambda$,
           $\mathcal U$ would have been an appropriate
           choice for $\mathcal U_\mu$ and if not chosen,
           then  $\mu\neq \mathcal U_{\mu} \sqsubset 
           \mathcal U$.  Since the sequence is 1-to-1,
            there is therefore a $\mu\in \Lambda\cap \lambda$
           such that $\mathcal U = \mathcal U_\mu$.  
           
           It should be clear that we can ensure
           that Martin's Axiom holds in the extension
           by making 
            minor adjustments to the choice of
             $\dot Q_\beta'$ for $\beta\notin \mathbf{E}$
in the  sequence
            $\langle \dot Q_\beta' : \beta < \kappa\rangle$
            together with 
 some additional bookkeeping, 
\end{proof}

 \section{Moore-Mrowka tasks}
 
 The Moore-Mrowka problem asks if every 
 compact space of countable tightness is sequential.
 A  space has countable tightness if  the closure
 of 
 a set  is equal to the union of   the closures
 of all its countable subsets. 
 A space is sequential providing that each
 subset is closed so long as
  it contains the
 limits of all its converging
 (countable) subsequences. 
To illustrate that a sequential space has
countable tightness, 
note  that 
  a space has countable tightness if a set
  is closed so long as  it contains the closures of
  all of its countable subsets.
 Say that a compact non-sequential space of countable
tightness is a Moore-Mrowka space. 
 
Results on the Moore-Mrowka problem have closely
resembled those of the S-space problem. 
In particular, there are proofs that PFA implies there are no
Moore-Mrowka spaces that have  many similarities to
the proof that PFA implies there are no S-spaces.  
While it is independent with CH as to whether
 Moore-Mrowka spaces exist \cite{chMM}, it is known
 that $\diamondsuit$ implies there are
 (Cohen indestructible) Moore-Mrowka spaces
 of cardinality $\aleph_1$ \cite{Ostaszewski}.
In addition, $\diamondsuit$ implies there is a separable
compact space of countable tightness with cardinality 
$2^{\aleph_1} $ (greater
than $\mathfrak c$) \cite{Fedorchuk}. It is also known that
the addition of $\aleph_2$ Cohen
reals over a model of $\diamondsuit + \aleph_2<2^{\aleph_1}$
results in a model in which there is a compact separable
space of countable tightness that has
  cardinality greater than $\mathfrak c$
  \cite{DowCohen}. 
  Of course these spaces are Moore-Mrowka spaces
  since
  every separable sequential
space has cardinality at most $\mathfrak c$.
 
 Here are two
open problems and a third that we solve in
the affirmative in this section.

\begin{question}
 Is it consistent with $\mathfrak c >\aleph_2$ 
 that every compact space of countable tightness is sequential?
\end{question}

\begin{question}
 Is it consistent with $\mathfrak p > \aleph_2$ 
 that there is a Moore-Mrowka space?
\end{question}

\begin{question}
 Is it consistent with $\mathfrak c > \aleph_2$ 
 that every  separable Moore-Mrowka space has
  cardinality at most $\mathfrak c$?
\end{question}
 
 A Moore-Mrowka task mentioned in the title of
 the section is  similar to an S-space task. The difference
 will be that rather than using the poset
  $Q(\mathcal U, C)$ to force  an uncountable discrete subset,
   we will hope to force an uncountable (algebraic) free sequence.
   We define these notions and indicate their relevance.

\begin{definition}
  A sequence $\{ x_\alpha : \alpha\in\omega_1\}$ is a 
  free sequence\label{freesequence}
   in a space $X$ if, for every $\delta<\omega_1$,
   the initial segment $\{ x_\alpha : \alpha \in \delta\}$ 
   and the final segment $\{ x_\beta : \beta\in \omega_1\setminus \delta\}$
   have disjoint closures. 
   
   A sequence $\{ x_\alpha, U_\alpha, W_\alpha : \alpha\in \omega_1\}$
   is an algebraic free sequence  in a space $X$ providing
   
\begin{enumerate}
 \item $x_\alpha\in U_\alpha$ and $W_\alpha$ are open sets with
  $\overline{U_\alpha}\subset W_\alpha$, 
  \item for every $ \alpha < \delta  \in\omega_1$,
  $x_\delta\notin W_\alpha$ and 
   there is a finite $H\subset   \delta{+}1$ such
   that $\{ x_\eta : \eta \leq \delta\}  \subset
    \bigcup\{ U_\beta : \beta\in H  \}$. 
\end{enumerate}
\end{definition}

Free sequences were introduced by Arhangelskii. 
Algebraic free sequences were introduced by Todorcevic
 in a slightly different formulation. The advantage
 of an algebraic free sequence is that the only reference
 to the (second order) closure property is with the
 pairs $U_\alpha, W_\alpha$. 
 If $\{ x_\alpha, U_\alpha, W_\alpha : \alpha\in\omega_1\}$
 is an algebraic free sequence, then 
 the set $\{ x_{\alpha+1} : \alpha < \omega_1\}$
 is a free sequence. This follows from the fact that
 for all $\delta\in\omega_1$,   there is a finite
  $H\subset \delta+1$ satisfying that
   $\{ x_\alpha : \alpha \leq \delta\} 
   \subset U_H = \bigcup\{ U_\alpha :\alpha\in H\}$
   and $\{ x_\beta : \delta< \beta \in\omega_1\}$ is disjoint
   from $W_H = \bigcup\{ W_\alpha : \alpha \in H\}$.
   The free sequence property now
   follows from the fact that
    $U_H$ and $X\setminus W_H$ have disjoint
   closures. This was crucial in Balogh's proof \cite{Balogh}
   that PFA implies there are no Moore-Mrowka spaces.

\begin{proposition}[\cite{Arh78}]
A compact space has countable tightness if and only
if it contains no uncountable free sequence. 
\end{proposition}

\begin{definition}
 A sequence $\mathcal A = \{ A_\alpha : \alpha \in \omega_1\}$
 is a Moore-Mrowka\label{MMT}
  task if, for all $\alpha\in\omega_1$,
 $\alpha\in A_\alpha \subset \alpha+1$, and
 
\begin{enumerate}
 \item for all $\beta < \alpha $ there is a $\gamma$ such
 that  $A_\gamma \cap \{\beta,\alpha\} = \{\alpha\}$,  and
\item   for all uncountable $A\subset\omega_1$, there is a
 $\delta\in\omega_1$ such that for all $\beta\in \omega_1\setminus
 \delta$, $(A\cap \delta) \cap \bigcap_{\gamma\in H}
  A_\gamma $ is not empty
 for all finite $H\subset \{ \gamma : \beta\in A_\gamma\}$.
\end{enumerate}
\end{definition}

The idea behind a Moore-Mrowka task is that we
identify $\omega_1$ with a set of points in space $X$
and so that there is
a  collection $\{ U_\alpha , W_\alpha : \alpha\in\omega_1\}$
that is a neighborhood assignment for those points. 
For each $ \alpha$,  $\overline{U_\alpha}\subset W_\alpha$
and $W_\alpha\cap\omega_1$
 is also contained in $\alpha+1$.
Then we set $A_\alpha = U_\alpha \cap \omega_1$.
Condition (1) is trivial to arrange but 
condition (2)  
is a $\diamondsuit$-like condition. A distinction with
S-space task is that the non-existence of a
Moore-Mrowka task extracted from a compact space
of countable tightness does
not imply that the space is sequential. The similarity
with S-space task is that we will use a Moore-Mrowka
task to generically introduce an algebraic free sequence.

\begin{definition}
 Let $\mathcal A=\{ A_\alpha : \alpha\in\omega_1\}$
 be a Moore-Mrowka task and let $C\subset\omega_1$ be  
 a cub. The poset $\mathscr M (\mathcal A, C)$ is the set
 of finite subsets of $\omega_1\setminus \min(C)$
  that are separated by
  $C$. 
 For each $H\in \mathscr M(\mathcal A,C)$ and each $\beta\in H$,
   let $A(H,\beta) $ be the intersection of
   the family $\{ A_\gamma : \gamma\in H, \ \beta\in A_\gamma \}$. 
  We define $H < K$ from $\mathscr M (\mathcal A, C)$ 
 providing  $H\supset K$ 
and for each $\alpha\in H\cap \max( K)$,
  $\alpha\in A(K,\min(K\setminus \alpha))$. 
\end{definition}

\begin{lemma}
 Let $\lambda\in \mathbf{E}$ and let 
 $\dot R_0$ be a $P_{\lambda}(I_\lambda) $-name 
 that is forced by $P_\lambda$ to be
  ccc poset. Let $\dot R$ be a $P_\lambda$-name of a
   ccc poset such $\mathbf{1}_{P_\lambda}$ forces 
    that $\dot R_0\subset_c \dot R$.
 Assume that $\mathcal A = \{\dot A_\xi : \xi\in\omega_1\}$
 is a sequence of $P_\lambda(I_\lambda) \star \dot R_0$-names 
 of subsets
 of $\omega_1$ such\label{step3}
 that $P_{\lambda}\star \dot R$ forces 
 that $\mathcal A$ is a Moore-Mrowka   task.  
 Then the $P_{\lambda+2}$-name 
 $\mathscr M (\mathcal U,\dot C_\lambda)$
 satisfies that $P_{\lambda+2}$ forces that $  \dot R\times 
 \mathscr M (\mathcal U,\dot C_\lambda)$ is ccc.
 \end{lemma}

\begin{proof}
    The   proof proceeds much as it did
   in Lemma \ref{step1} 
   and Corollary \ref{step2} for S-space tasks.
To show that a poset of the form $\mathscr M(
\mathcal A, C)$ is ccc, it again suffices to prove
that, for each $n\in\omega$, there is no 
uncountable antichain consisting of pairwise
disjoint sets of cardinality $n$. So we consider
an arbitrary family of pairwise disjoint sets of
cardinality $n$.   
   Fix  $P_{\lambda+2}\star\dot R$-names 
   $\{ \dot H_\xi : \xi\in\omega_1\}$ 
     for a set of pairwise disjoint elements of
      $\mathscr M (\mathcal A, \dot C_\lambda)\cap [\omega_1]^n$. 
      Following Lemma \ref{step1}, we may assume
      that, for each $\xi\in\omega_1$, it is forced
      that $\xi<\min(\dot H_\xi)$ and that
       $\{\xi\}\cup \dot H_\xi$ is also separated by 
       $\dot C_\lambda$.
      We prove
      that no condition forces this to be an antichain. 
  
   Let $M$ be a countable elementary submodel
      containing all the above and let
      $p_1\in P_{\lambda+2}$ be chosen as
      in Lemma \ref{Cohengeneric} so that
       $p_1$ is $(M,P_{\lambda+2})$-generic 
       and so that $p_1(\lambda)\in M$.      
      Let
      $p_1\restriction\lambda \in
      G_\lambda$ be a $P_\lambda$-generic filter 
      and pass to the extension $V[G_\lambda]$.
      Let $R=\val_{G_\lambda}(\dot R)$ and let
      $G_1\subset \mathcal C_{\omega_1}$ so that
       $
       p_1\restriction{\lambda+1}\in
       G_\lambda\star G_1$ is $P_{\lambda+1}$-generic.
       Let $\delta=M\cap \omega_1$.
       We 
       will prove that  $p_1$ forces that
       $\dot H_\delta$ is   
 compatible with some element of $\{  \dot H_\eta : \eta\in \delta\}$.
 
 For each $\zeta\in\omega_1$, let, in $V[G_\lambda]$,
 $\dot J_\zeta$ denote the 
  $R$-name for the set $  \{ \gamma : 
 \zeta \in \dot A_\gamma\}$
 and,
 for each finite 
 $F\subset \omega_1$, also let 
 $\dot A_F$ denote
 the $R$-name for $\bigcap_{\gamma\in F}\dot A_\gamma$. 
 We leave the reader to check that it
  suffices to prove that $p_1$ forces that
  for each finite $F\subset \dot J_{ \min(\dot H_\delta)}$, 
  there is an $\eta <\delta$ such that $\dot H_\eta\subset
  \dot A_F$. For each $\zeta\in\omega_1$ and
  finite $F\subset \omega_1$,
   we will let $J_\zeta$ and $A_F$ denote
  $\val_{G_R}(\dot J_\zeta)$ and    
  $\val_{ G_R}(\dot A_F)$ respectively.
  Also, for the remainder of the proof we will
  treat each $\dot H_\xi$ as the
  canonical
   $R\times (Q_\lambda\star \dot Q_{\lambda+1})$-name
   obtained from the evaluation of the original
    $P_{\lambda+2}\star \dot R$-name by $G_\lambda$.
    For each $\xi\in\omega_1$ and $H\in [\omega_1]^n$,
    let $\Gamma_\xi(H)$ be the (possibly empty) set
    of conditions in $R\times
    (Q_\lambda\star \dot Q_{\lambda+1})$
    that force $H$ to equal $\dot H_\xi$.

  We need an updated version of $\omega_1$-full.  
  Say that a countable set $B$, in $V[G_\lambda][  
  G_R ]$,
   is $\mathcal A$-large if 
  there is a $\gamma\in\omega_1$ such that 
   $B\cap A_F\neq\emptyset$ for all $\beta\in \omega_1\setminus
   \gamma$ and finite $F\in J_\beta$. We may interpret
   this as that $\overline{B}$ contains $\omega_1\setminus \gamma$.

    For  $\xi\in\omega_1$ and 
   $(r,p)\in R
   \times  ( Q_{\lambda} \star
 \dot  Q_{\lambda+1})$, 
    let $\Gamma_\xi(H,(r,p ))$ 
     be the set   of conditions in $\Gamma_\xi(H)$
 that are below $(r,p)$.
     In other words, 
     $\Gamma_\xi(H,(r,p))$ is not empty
     if and only if $(r,p   )\not\Vdash H\neq \dot H_\xi$.  
Similarly, for each $0<i<n$ and $H\in [\omega_1]^{i}$,
 let  $\Gamma_\xi(H,(r,p)) = \bigcup\{ \Gamma_\xi(H\cup \{\eta\},
 (r,p)) :
  \eta\in \omega_1\}$. 
For $H\in [\omega_1]^n$, say that
 $\Gamma_\xi(H,(r,p))$ is  full if 
 $\Gamma_\xi(H,(\bar r,p))$ is not empty
 for all $\bar r\leq r$. 
  For $0<i<n$ and $H\in [\omega_1]^{n-i}$, say that
   $\Gamma_\xi(H,(r,p))$ is  full if 
   there is a $R $-name $\dot B$
   that is forced to be
   an  $\mathcal A$-large set of $\eta\in\omega_1$
   and, for each $\eta$ and $s \Vdash \eta\in \dot B$,
  $\Gamma_\xi(H\cup \{\eta\}, (s,p))$ is
      full.

\begin{claim}
Suppose that $\xi,r,p\in M[G_\lambda]$ and that
$\Gamma_\xi(\emptyset , (r,p))$ is full.
Suppose also that\label{3-1} $\bar r\in R$
forces that 
   $F$ is a finite subset
   of $\dot J_{\zeta}$ for some $\delta\leq \zeta\in\omega_1$.
   Then 
  there are $(s,q), H\in M[G_\lambda]$ and $\bar s<\bar r$
    such that 
   
\begin{enumerate}
 \item $(s,q) < (r,p)$ in $R\times (Q_\lambda\star \dot Q_{\lambda+1})$,
 \item $\bar s< s $,
 \item $\bar s\Vdash H\subset \dot  A_F$,
 \item $(s, q)\Vdash  \dot H_\xi = H$.
\end{enumerate}
\end{claim}

\bgroup
\def\proofname{Proof of Claim:\/}

\begin{proof} 
There is an   $R\times Q_\lambda$-name
$\dot B_0\in M[G_\lambda] $ that is forced to be a
$\mathcal A$-large  subset of $\delta$
and witnesses that
$\Gamma_{\xi}(\emptyset,(r,p))$ is full.
Therefore there are $\eta<\delta$ and $r' < \bar r$ 
  such that $ \bar r_1 \Vdash \eta\in \dot B_0
 \cap A_F$. There is no loss to assuming,
 by elementarity, that
 $ \bar  r_1   $ extends some $ r_1 \in M[G_\lambda]$ 
 such that $ r_1 \Vdash \eta\in \dot B_0$.
 Since $r_1\Vdash \eta\in \dot B_0$,
  we have that $\Gamma_\xi(\{\eta\},
   (r_1,p))$ is full. Following a recursion
   of length $n$, there is an $\bar r_n < \bar r$
   in $R$, an $H=\{\eta_1,\ldots,\eta_n\}\in
    M[G_\lambda]$, and an $\bar r_n < r_n\in M[G_\lambda]$
    such that  $\bar r_n\Vdash H\subset A_F$ and
    $\Gamma_\xi(H, (r_n,p))$ is full.
    Since $\bar r_n < r_n$ , 
       $\Gamma_\xi(H, (\bar r_n,p))$ is not empty.
       Therefore there is a pair
        $(\bar s, \bar q) < (\bar r, p)$ forcing
        that $H=\dot H_\xi$. 
 By elementarity, since $\xi,H,p\in M[G_\lambda]$,
 the set of $\{ s\in R\cap M : (\exists q) ( (s,q) < (r_n,p) \ \&
 (s,q)\Vdash H=\dot H_\xi)\}$ is predense below $r_n$.
 Therefore 
  there is an $(s,q)<(r_n,p)\in M[G_\lambda]$ with
  $s\not\perp \bar r_n$ such that
   $(s,q)     \Vdash H=\dot H_\xi$. Let $\bar s$ be any
   extension of $s,\bar r_n$.
 \end{proof}

\begin{claim}
 For every $(r,p) \in R\times (Q_\lambda\star    \dot Q_{\lambda+1})$, 
  there is a $\delta$ so\label{3-2}
and a $r_0<r$ such  that 
 $\Gamma_\delta(\emptyset, (r_0,p))$ is  full.
\end{claim}

\begin{proof}
 Let $(r,p)\in 
 M_0$ be a countable elementary submodel of 
 $H(\kappa^+)[G_\lambda]$
 so that $\{ \mathcal A,  R, P_{\lambda+2}  \}\in M_0$. 
Choose any $(\bar r,\bar p)<(r,p)$ that is an
  $(M_0 , R\times (Q_\lambda\star    \dot Q_{\lambda+1}))$-generic condition.
   Let $\delta_0 = M_0\cap \omega_1$.
 Choose any continuous $\in$-chain 
  $\{ M_\alpha : 0<\alpha < \omega_1\}$
    of countable
  elementary submodels of $H(\kappa^+)[G_\lambda]$ such 
  that $\{M_0,(\bar r,\bar p) \}\in M_1 $.

  For each $\alpha\in\omega_1$, let $\delta_\alpha = 
   M_\alpha\cap \omega_1$.     Let $C^* $ be the cub
    $\{ \delta_\alpha :\alpha\in\omega_1\}$.
Choose   any extension $( r_n, p_n)$
of $(\bar r,\bar p)$ such that 
$\pi_1(  p_2( \lambda+1)) \subset C^*\cup\delta_0$,
and so that there is an $H
=\{\xi_1,\ldots,\xi_n\}\in [\omega_1]^n$ with
 $(r_n,  p_n)\Vdash H= \dot H_{\delta_0}$. 
     Of course this implies that 
$ \Gamma_{\delta_0}(H,( r_n,p) )
\supset \Gamma_{\delta_0}(H,(r_n,p_n))$ is 
actually full.
  Okay,
  then $H_{n-1}= \{\xi_1,\ldots, \xi_{n-1}\} $ is 
  in $M_{\alpha_n}$. 
 Let's take the $R$-name 
  $\dot E_{n-1}$ to the set of $(\eta, \tilde r) $ such 
  that $\Gamma_{\delta_0}(\{\eta\}\cup H_{n-1}, (\tilde r, p))$
  is full. The condition $r_n$ forces that
   $\dot E_{n-1}$ is uncountable. Since
    $\mathcal A$ is a Moore-Mrowka task
    in $V[G_\lambda\star G_R]$, $r_n$ forces
    that $\dot E_{n-1}
    \in M_{\alpha_n}$ contains an $\mathcal A$-large set.
    By elementarity and the fact that $r_n$ is $(M_{\alpha_n},
     R)$-generic, there is an $r_{n-1}$ in $M_{\alpha_n}$
     that forces $\dot E_{n-1}$ contains an
     $\mathcal A$-large set. Therefore, for such
     an $r_{n-1}\in M_{\alpha_n}$, 
     we have that $\Gamma_{\delta_0}(H_{n-1}, 
      (r_{n-1},p))$ is full. This recursion continues
      as above and for each $ i<n$,
       there is an $r_i \in M_{\alpha_i}$ such that 
    $\Gamma_{\delta_0}(\{\xi_j : j<i\}, (r_i,p))$ is  
    full. Setting $\delta  = \delta_0$, 
    this completes the proof of the Claim.
\end{proof}
\egroup

Following the proof of Corollary \ref{step2}  we
can complete the proof using
that $p_1$ satisfied the conclusion of 
 Lemma \ref{Cohengeneric}.
 Using Claim \ref{3-2}, it follows from
  Claim \ref{3-1} that in $V[G_\lambda][G_R]$,
  for each $\delta\leq \zeta\in\omega_1$
  and  finite 
  $F\subset    J_\zeta $,  the set
   $\mathcal W_F$ consisting of 
   those 
  $p\in M[G_\lambda]\cap (Q_\lambda\star \dot Q_{\lambda+1})$   
  for which there is a $\bar s\in G_R$ 
  and $\xi\in\delta$ such
  that $(\bar s, p)\Vdash \dot H_\xi \subset A_F$,
   is a dense subset of $M[G_\lambda]
   \cap (Q_\lambda\star \dot Q_{\lambda+1})$. 
   By the genericity of 
   $((G_1)\restriction\delta)\star 
    (p_1^\uparrow{}_\lambda)$ over the model 
     $V[G_\lambda\star R]$ 
          as in Lemma 
    \ref{Cohengeneric}, 
    it meets $\mathcal W_F$.  It  follows that $p_1$ forces
     that there is a $\xi\in \delta$ such
     that $\dot H_\xi \subset A_F$. 
     Applying this fact to
 $\zeta = \min(H_\delta)$ completes the proof.   
  \end{proof}

Now we show that Moore-Mrowka tasks will arise 
that will allow us to prove there is a minor additional
condition that we can place on the construction 
of $P_{\kappa+\kappa}$ 
(assuming an extra $\diamondsuit$-principle) that will
force there 
are no separable
Moore-Mrowka spaces of cardinality greater than $\mathfrak c$.
Let $S^\kappa_1$ denote the set of $\lambda\in \kappa$
that have cofinality $\omega_1$. We will
 assume there is a $\diamondsuit(
 S^\kappa_1)$-sequence.

We begin with this Lemma.

\begin{lemma}[${\mathfrak c}^{<\mathfrak c} =\mathfrak c$]
 Let $X$ be a separable Moore-Mrowka space
 of cardinality greater than $\mathfrak c$.   
 Let $X\in M$ be an elementary submodel of\label{kappaMM}
  $H(\theta)$ for some sufficiently large $\theta$
  such that $|M|=\mathfrak c$ and $M^\mu \subset M$
  for all $\mu<\mathfrak c$. 
For any   point $z\in X\setminus M$ there is a sequence
  $\{ B_\eta : \eta < \mathfrak c\}$ of countable subsets
  of $M\cap X $ satisfying, for all $\eta<\zeta<\mathfrak c$,
\begin{enumerate}
 \item $\overline{B_\eta}$ contains $B_\zeta\cup\{z\}$ 
\item for all $A\subset M\cap X$ with $z\in \overline{A}$,
there is an $\alpha
  < \mathfrak c$ such that
  $\overline{A}$ contains $B_\alpha$. 
 \end{enumerate}
\end{lemma}

\begin{proof}
 Since $X$ is separable, we can let $B_0\in M$ be any
 countable dense subset.  Fix an enumeration
  $\{ S_\xi : \xi <\mathfrak c\}$ of all the countable subsets
  of $M\cap X$ that have $z$ in their closure. 
  Let $\mathcal W\in M$ be a base for the topology.
 Assume we have 
 chosen $\{ B_\xi : \xi < \eta\}$ for some $\eta< \mathfrak c$.
 Assume, by induction, that $B_\xi$ is also a subset
 of $\overline{S_\xi}$. The set $\overline{S_\eta}
 \cup \{ \overline{B_\xi} : \xi <\eta\}$ is an element
 of $M$ and every member   contains $z$. 
 Let $K_\eta$ denote the intersection of this family.
 Choose any  neighborhood  $U  \in \mathcal W$
 of $z$. Since  $z\in W\cap K$, it follows from
 elementarity that $M\cap W\cap K_\eta$ is non-empty.
  Therefore, $z$ is in the closure of some countable
   $B_\eta \subset M\cap K_\eta$.
   This completes the inductive construction of the
   family. We simply have to verify that property
    (2) holds.  Let $z\in \overline{A}$ for some
     $A\subset M\cap X$. By countable tightness,
      there is an $\eta$ such that $S_\eta\subset A$.
      Therefore $\overline{A}\supset B_\eta$.
\end{proof}

\begin{remark}
 A compact separable space of cardinality at most
 $\mathfrak c$ will have a $G_\delta$-dense set
 of points of character less than $\mathfrak c$.
 Therefore,
 in a model with $\mathfrak p = \mathfrak c$, 
 any such space has the property that
 the sequential closure of any subset is countably
 compact. In particular, in such a model
 a Moore-Mrowka space necessarily
 has weight at least $\mathfrak c$ and
 will have a countably compact subset that
 is not closed. A space is said to be C-closed
 if it has no such subspace, see \cites{Ismail,Cclosed}. 
\end{remark}

\begin{definition}
Say that  a sequence $\langle y_\alpha, U_\alpha , 
W_\alpha : \alpha < \kappa\rangle$
 is a $\kappa$-MM sequence of a space $X$ if
 
\begin{enumerate}
 \item $U_\alpha, W_\alpha$ are open 
  in $X$ and 
 $y_\alpha\in U_\alpha \subset \overline{U_\alpha}\subset W_\alpha$,
  \item $y_\gamma\notin U_\alpha$ for all $\alpha< \gamma\in\kappa$,
   \item for all $\beta < \alpha<\kappa$, $U_\gamma\cap \{ y_\beta, y_\alpha\} =
    \{y_\alpha\}$ for some   $\alpha\leq \gamma
   \in \kappa$,
    \item for every $A \subset  \kappa $, there is a countable 
    $B\subset A$  and a $\gamma<\kappa$ such that  
    the closure of
       ${\{ y_\alpha : \gamma<\alpha <\kappa\}}$
       is either contained in 
       the closure of $ {\{ y_\beta : \beta\in B\}}$ 
       or is disjoint from the closure of 
  $ {\{y_\alpha : \alpha \in A\} }$.
\end{enumerate}
\end{definition}

\begin{theorem}
 Let   $\langle 
  P_\alpha , \dot Q_\beta : \alpha \leq \kappa+\kappa,
  \ \beta <\kappa+\kappa\rangle$
  be an iteration sequence  in the sense of Theorem \ref{mainthm}.
  In particular, assume that for all $\mu<\kappa$
  there is a\label{reflectMM}
        $\q_\mu\in   \mathcal Q(\mu,\mu)$
    satisfying that $P_{\kappa+\lambda}$ is
    equal to $P^{\q_\mu(\kappa)}_{\kappa+\mu}$.     

     Let $\dot X$ be a 
   $P_{\kappa+\kappa}$-name of a compact
   separable space of countable tightness. Assume
   also  that $\langle
   \dot y_\alpha, \dot U_\alpha, \dot W_\alpha :
   \alpha <\kappa\rangle$ is forced
   to be a $\kappa$-MM sequence of $\dot X$.   
      Then there is a
   cub $C_{\dot X}\subset \kappa$ such that for
  each
    $\lambda\in C_{\dot X}\cap  S^\kappa_1$, 
    there is an injection $f_\lambda :\omega_1 
    \rightarrow  \lambda$ such
    that 
     $\mathcal A = \langle  \dot A_\eta : \eta < \omega_1\rangle$,
      where $\dot A_\eta  = 
     \{ \xi : y_{f_\lambda(\xi)}
     \in \dot U_{f_\lambda(\eta)}\}$,
 is
      forced by $P_{\lambda+\lambda}^{\q_\lambda}$
      to be a Moore-Mrowka task.
       
\end{theorem}

\begin{proof} 
We may assume,
 since it is forced to be compact and separable,
  that $\dot X$ is a $P_{\kappa+\kappa}$-name
of a closed subspace of $[0,1]^\kappa$. 
Let $G$ be a $P_{\kappa+\kappa}$-generic filter 
so that we may make some observations about
 $\dot X$ and the   $\kappa$-MM sequence
  $\langle
     y_\alpha,   U_\alpha,   W_\alpha :
   \alpha <\kappa\rangle$. There is a point
    $z\in \val_G(\dot X)$ that is a
   $\kappa$-accumulation point of 
   $\{y_\alpha :\alpha\in \kappa\}$. 
   We check that $z$ is the unique such point.
   If $U,W$ are open neighborhoods of $z$
   with $\overline{U}\subset W$, 
   then $A= \{ \alpha \in \kappa : 
    y_\alpha\in U\}$ is cofinal in $\kappa$.
     By condition (4) of the $\kappa$-MM property,
      there is a countable $B\subset A$
      so that the closure of $\{ y_\beta : \beta\in B\}$
      contains $\{ y_\alpha : \sup(B) <\alpha<\kappa\}$.
      It thus follows that 
      that
      $\{ y_\alpha : \sup(B) <\alpha<\kappa\}$
      is contained in $W$ and shows that
       $X\setminus W$ contains no $\kappa$-accumulation
       points of $\{ y_\alpha : \alpha\in \kappa\}$. 
       Now assume that $z$ is in the closure
       of $\{ y_\beta : \beta\in A\}$ for some 
       $A\subset \kappa$.   
       Since the second clause
       of condition (4) of
the $\kappa$-MM property fails, it follows that 
there is a countable $B\subset A$ such that
 the closure of $\{ y_\beta : \beta\in B\}$
 contains a final segment of $\{ y_\alpha : \alpha
 \in \kappa\}$.  
 We will be interested in the subspace
  $X_\lambda = \{ x\restriction\lambda :  x\in X\}$ of
   $[0,1]^\lambda$. Since this space is 
   a continuous image of $X$, it also has countable
   tightness.
 Let $\dot z$ be a canonical
 $P_{\kappa+\kappa}$-name for $z$.

 Let $M\prec H(\kappa^+)$ so that 
 $\sup(M\cap \kappa)=\lambda
 \in S^\kappa_1$ and $M^\omega\subset M$.  
 We note that it follows from Corollary \ref{noreals},
 and the fact that $P_{\kappa+\kappa}/P_\kappa$ is ccc,
 that every countable subset of $M\cap \kappa$ in
 $V[G]$ has a name in $M$.  
 Assume also that $\dot z,\dot X, P_{\kappa+\kappa}$
 and the $\kappa$-MM sequence are elements of $M$.
  Choose any continuous $\in$-increasing sequence
  $\{ M_\eta : \eta\in \omega_1\}$ of countable elementary
  submodels of $M$ such that $Y_\lambda
   = \bigcup\{ M_\eta\cap \lambda : \eta\in\omega_1\}$ 
   is cofinal in $\lambda$. Define $f_\lambda$
   so that $f_\lambda(\eta) = \sup(M_\eta\cap \lambda)$. 
   It should be clear that to show that
   $\mathcal A$, as in the statement of the Theorem,
    is forced by
    $P^{\q_\lambda}_{\lambda+\lambda}$
    to be a Moore-Mrowka task 
  it is sufficient  to check that condition (2) of
   Definition \ref{MMT} is forced to hold. Let
    $\dot A$ be any 
    $P^{\q_\lambda}_{\lambda+\lambda}$-name
     of an uncountable subset of $\omega_1$. 
    We may regard $P^{\q_\lambda}_{\lambda+\lambda}$
    as a complete subposet of $P_{\kappa+\kappa}$ 
    and so consider $\val_G(\dot A)$ in $V[G]$.
    In the space $X_\lambda$, it is clear
    that $z\restriction \lambda$ is in the closure
    of the set $\{ y_{f_\lambda(\eta)} : \eta\in A\}$. 
    Therefore, there is a countable $B\subset A$
    such that 
    $z\restriction \lambda$ is in the closure
    of the set $
   \vec y (f_\lambda(B)) = \{ y_{f_\lambda(\eta)} : \eta\in B\}$. 
    Now $B$ 
    is a countable subset of $M\cap \lambda$,
     and so there is a $P^{\q_\lambda}_{\lambda+\lambda}$-name
     $\dot B$ in $M$ such that $\val_{G}(\dot B) $ is $B$.
     Now we can apply elementarity (using that
      $f_\lambda\restriction B\in M$) and 
     observe that $\dot z$ is forced to be in the closure
     of $\{ \dot y_{f_\lambda(\beta)}  : \beta\in \dot B\}$. 
     Moreover, by elementarity and the $\kappa$-MM
     property, there is a $\gamma\in \kappa\cap M$
     such that the closure
     of $\vec y(f_\lambda(\dot B))$   is forced to
     contain $\{ \dot y_\alpha : \gamma < \alpha < \kappa\}$. 
     For each $\gamma<\alpha <\kappa$,
        $\vec y(f_\lambda(\dot B))$ is forced to meet
        $\bigcap_{\zeta\in H}\dot U_\zeta$
         for all 
        finite $H\subset \{ \zeta : \alpha\in \dot U_\zeta\}$. 
     Of course there is an $\delta\in \omega_1$ such 
     that $\gamma < f_\lambda(\delta )$. This completes
     the proof that, for all $\beta\in \omega_1\setminus \delta$,
      $\dot A\cap \delta  $ is
     forced to meet      $\bigcap_{\zeta\in H}\dot A_\zeta$
     for all finite $H\subset \{ \zeta : \beta \in \dot A_\zeta\}$.
\end{proof}

\begin{theorem} It is consistent with Martin's Axiom
 and $\mathfrak c >\aleph_2$ that there are no
 S-spaces and that compact separable spaces of countable
 tightness have cardinality at most $\mathfrak c$.
\end{theorem}

\begin{proof}
Let $\kappa>\aleph_2$ be a regular cardinal in 
a model of GCH. 
 Using an iteration sequence as in Theorem \ref{mainthm},
 it follows from Theorem \ref{reflectMM} 
 and Lemma \ref{kappaMM}
 that it suffices to ensure that  for each $\dot X$ 
 and
 $\kappa$-MM-sequence as in Theorem \ref{reflectMM},
 there is a $\lambda\in C_{\dot X}\cap S^\kappa_1$
 so that $I_\lambda$ is chosen suitably and so that
  $\dot Q_{\kappa+\lambda}$
 is chosen to be $\mathscr M(\mathcal A,\dot C_\lambda)$
 for a sequence $\mathcal A$ as identified in 
 Theorem \ref{reflectMM}. This is a somewhat
 routine application of $\diamondsuit(S^\kappa_1)$.
 
 Since $S^\kappa_1$ is stationary, we may assume
 that $\diamondsuit(S^\kappa_1)$ holds in $V$.
 There are many equivalent formulations
 of $\diamondsuit(S^\kappa_1)$ and we choose this one:
  There is a sequence $\langle h_\alpha : \alpha 
  \in S^\kappa_1 \rangle $ satisfying
  
\begin{enumerate}
 \item for each $\alpha\in S^\kappa_1$, $h_\alpha :
 \alpha\times \alpha\rightarrow \alpha$ is  a function,
 \item for all functions $h:\kappa\times\kappa\rightarrow \kappa$,
 the set $\{ \alpha\in S^\kappa_1 : h_\alpha\subset h\}$ 
 is stationary.  
\end{enumerate}

We will also have to recursively define our
sequence $\mathscr I = \{ I_\gamma : \gamma\in 
\mathbf{E}\}$ since special choices will have to be
made for indices in $S^\kappa_1$ and which, due
to conditions (3) and (4) impact all the subsequent
choices. 
  To assist with
the condition (4) of the requirements on $\mathscr I$,
we choose
an enumeration $\{ J_\xi : \xi\in\kappa\}$
   of $[\kappa]^{\aleph_1}$ as follows.
Let $D\subset  \kappa$ be 
a cub consisting of 
$\lambda $ such that 
 $\mu+\mu^{\aleph_1}<\lambda$ for all $\mu<\lambda$.
 For each $\mu\in D $,
    the list $\{ J_\xi : \mu\leq \xi < \mu+\mu^{\aleph_1}\}$ 
    is an enumeration of $[\mu]^{\aleph_1}$.
     
Say that a sequence $\mathscr I_\lambda =
\{ I_\gamma : \gamma\in 
 \mathbf{E}\cap \lambda\} \subset [\lambda]^{\leq \aleph_1}
 $ is an acceptable sequence if 
 it satisfies the
 properties (1), (2), and (3) that we assume for the sequence
  $\mathscr I$ in section 2,
  and, it also satisfies that, for each $\xi<\mu\in \lambda$
  such that $\mu+\mu^{\aleph_1}<\lambda$, there is a $\zeta\in 
  \mathbf{E}\cap \mu+\mu^{\aleph_1}$
  such that $J_\xi \subset I_\zeta$.
  If $  \{ \mathscr I_\lambda : \lambda \in D\}$ is
 an increasing sequence of acceptable sequences,
   then the union,
   $\mathscr I$, satisfies the requirements of section 2.
   Similarly, once we have chosen an acceptable sequence
    $\mathscr I_\lambda$, 
     we will assume that the sequence 
     $\langle P_\alpha , \dot Q_\beta : \alpha\leq \lambda, 
     \beta < \lambda \rangle$ is defined as
     in Definition \ref{conditions} using the sequence
      $\mathscr I_\lambda$. 
   
   In a similar fashion, we relativize the definition
   of $\mathcal Q(\lambda,\mu)$ from Definition 
    \ref{defineQ}. Given an acceptable sequence
     $\mathscr I_\lambda$, say that a sequence
      $\q' = \{ \dot Q'_\beta : \beta < \lambda\} \in H(\kappa)$
      is $\mathscr I_\lambda$-suitable providing
      (as in Theorem \ref{mainthm}),
      by induction on $\beta <\lambda$, 
      $\dot Q^\lambda_\beta(\q)  = \dot Q'_\beta$ is a
      $P^\lambda_{\lambda +\beta}(\q)$-name of a ccc poset,
      where $P^\lambda_\alpha(\q) = P_\alpha$ for $\alpha\leq\lambda$
      and, for $\beta>0$, $P^\lambda_{\lambda+\beta}(\q)$ 
      is the usual poset from the iteration sequence
       $\langle P^\lambda_{\alpha}(\q)
       , \dot Q^\lambda_{ \zeta}(\q):  \alpha\leq \beta, \zeta < \beta
       \rangle$.
    
 Let $f$ be any function from $\kappa$ onto
 $H(\kappa)$.  We recursively choose our sequences
  $\{\mathscr I_\lambda : \lambda\in D\}$ and
   $\{ \dot Q_\gamma' : \gamma \in \kappa\}$. 
   The critical inductive assumptions are, for $\lambda\in D$,
\begin{enumerate}
 \item  $\mathscr I_\lambda$ extends $\mathscr I_\mu$
 for all $\mu\in D\cap \lambda$,
 \item $\mathscr I_\lambda$ is acceptable,
 \item $\{ \dot Q_\gamma' : \gamma<\lambda\}$ is 
 $\mathscr I_\lambda$-suitable.  
\end{enumerate}
  
  Now let $\lambda\in D$ and assume we have
  constructed, for each $\mu\in D\cap \lambda$,
  $\mathscr I_\mu$  and
  and $\{\dot Q_\gamma' :  \gamma <\mu\}$.
  If $D\cap \lambda$ is cofinal in $\lambda$,
   then we simply let $\mathscr I_\lambda
    =\bigcup\{\mathscr I_\mu : \mu\in D\cap \lambda\}$
    and there is nothing more to do. 
    Otherwise, let $\mu$ be the maximum element
    of $D\cap \lambda$.
    \medskip
    
    Case 1:
 $\mu\notin S^\kappa_1$. First choose any 
 acceptable $\mathscr I_\lambda\supset \mathscr I_\mu$. 
     Choose $\{ \dot Q'_\beta : \mu\leq\beta < \lambda\}$ by
     induction as follows. For $\mu<\beta \notin \mathbf{E}$,
       let $\q$ 
    denote $\{ \dot Q'_\gamma : \gamma < \beta\}$.
Let $\zeta<\kappa$ be minimal so that $\dot Q'_\beta = 
f(\zeta) $ is a
 $P^\mu_{\mu+\beta}(\q)$-name of a  ccc poset  that is not in the list
  $\{ \dot Q'_\gamma : \gamma <\beta\}$. 
  For $\mu \leq \beta\in \mathbf{E}$, choose, if possible
  minimal $\zeta  <\kappa$ so that  $f(\zeta)$ is
  equal to $Q(\mathcal U, \dot C_\beta)$ for some
  S-space task that is not yet handled
  and let $\dot Q'_\beta = f(\zeta)$.
   Otherwise, let $\dot Q'_\beta = \mathcal C_\omega$.
   \medskip
   
   The verification of the inductive hypotheses in Case 1 is routine.
   We also note that if the induction continues to $\kappa$,
   then $P^\kappa_{\kappa+\kappa}(\{\dot Q'_\beta : \beta <\kappa\})$
   will force that there are no S-spaces and that Martin's Axiom
   holds. 
   \medskip
   
   Case 2: $\mu\in S^\kappa_1$. Let $\q$ denote
    $\{\dot Q'_\beta : \beta < \mu\}$. 
 Now we decode the element
 $h_\mu$ from the $\diamondsuit$-sequence.
 If there is any $(\alpha,\xi)\in \mu\times\mu$
 such that $f(h_\mu(\alpha,\xi))$ is not a
 $P^\mu_{\mu+\mu}(\q)$-name, then proceed as
 in Case 1. For each $\alpha\in\mu$, 
 if $f(h_\mu(\alpha,0))$ is not a name of
 a finite subset of $\mu$, then proceed as in Case 1,
 otherwise let $\dot F_\alpha = f(h_\mu(\alpha,0))$.
 Similarly, if there is an $\alpha\in\mu$ such that
  $f(h_\mu(\alpha,1))$ is not a name of a positive
  rational number, then proceed as in Case 1,
  otherwise let $\dot \epsilon_\alpha = f(h_\mu(\alpha,1))$.
  If there is an $\alpha\in \mu$ and a $\xi >1$
  such that $f(h_\mu(\alpha,\xi))$ is not a name
  of a element of $[0,1]$, then proceed as in Case 1,
  otherwise let 
  $$
\mbox{for}\ (\alpha,\xi)\in \mu\times\mu\ \ \ 
  \dot y_\alpha (\xi) =\begin{cases}
     f( h_\mu(\alpha,\xi+2)) & \mbox{if }\ \xi<\omega\\
     f( h_\mu(\alpha,\xi)) & \mbox{if }\ \omega\leq \xi<\mu   
     \end{cases}\ .$$
 It now follows that $\dot y_\alpha$ is a name of
 an element of $[0,1]^\mu$ and let 
 the name   $\{ x\in [0,1]^\mu : (\forall \beta\in \dot F_\alpha)
   | x(\beta)-\dot y_\alpha(\beta)| < \dot \epsilon_\alpha\}$
   be denoted by
 $\dot U_\alpha$. Now we ask if there is a
 function $f_\mu : \omega_1 \rightarrow \mu$
 as in Theorem \ref{reflectMM}. In particular,
  if there is an $I\in [\mu]^{\aleph_1}$ and such
  a function $f_\mu :\omega_1\rightarrow \mu$
  such that the sequence $\mathcal A = 
   \{ \dot A_\eta : \eta\in\omega_1\}$ as defined
   in the statement of Theorem \ref{reflectMM} satisfies
   that $P^\mu_{\mu+\mu}(\q)$ forces that 
    $\mathcal A$ is a Moore-Mrowka task 
    and each $\dot A_\alpha$ is
    a $P^\mu_{\mu+\mu}(\q)(I)\star \dot R_0$-name
    in the sense of Lemma \ref{step3}. 
    If all this holds, then choose an appropriate $I_\mu$ so
    that $I\subset I_\mu$ and define
    $\dot Q_\mu'$ to be $\mathscr M(\mathcal A,
     \dot C_\mu)$. For the remaining choices
     proceed as in Case 1.
  
  The construction of $P_{\kappa+\kappa} = 
   P^\kappa_{\kappa  +\kappa}(\q)$ where
    $\q= \{ \dot Q_\beta' : \beta < \kappa\}$ is complete. 
As explained at the beginning of the proof, it
follows from Lemma \ref{kappaMM}
and
Theorem \ref{reflectMM},
and that the fact that $D$ is a cub,
  that separable Moore-Mrowka spaces in this model
  have cardinality at most $\mathfrak c$.   
\end{proof}


\begin{bibdiv}

\def\cprime{$'$} 

\begin{biblist}
 
  \bib{ARS}{article}{
   author={Abraham, Uri},
   author={Rubin, Matatyahu},
   author={Shelah, Saharon},
   title={On the consistency of some partition theorems for continuous
   colorings, and the structure of $\aleph_1$-dense real order types},
   journal={Ann. Pure Appl. Logic},
   volume={29},
   date={1985},
   number={2},
   pages={123--206},
   issn={0168-0072},
   review={\MR{801036}},
   doi={10.1016/0168-0072(84)90024-1},
}
			 
 \bib{Avraham}{article}{
   author={Avraham, Uri},
   author={Shelah, Saharon},
   title={Martin's axiom does not imply that every two $\aleph _{1}$-dense
   sets of reals are isomorphic},
   journal={Israel J. Math.},
   volume={38},
   date={1981},
   number={1-2},
   pages={161--176},
   issn={0021-2172},
   review={\MR{599485}},
   doi={10.1007/BF02761858},
}
\bib{Arh78}{article}{
   author={Arhangel\cprime ski\u{\i}, A. V.},
   title={The structure and classification of topological spaces and
   cardinal invariants},
   language={Russian},
   journal={Uspekhi Mat. Nauk},
   volume={33},
   date={1978},
   number={6(204)},
   pages={29--84, 272},
   issn={0042-1316},
   review={\MR{526012}},
}
	
	\bib{Balogh}{article}{
   author={Balogh, Zolt\'{a}n T.},
   title={On compact Hausdorff spaces of countable tightness},
   journal={Proc. Amer. Math. Soc.},
   volume={105},
   date={1989},
   number={3},
   pages={755--764},
   issn={0002-9939},
   review={\MR{930252}},
   doi={10.2307/2046929},
}
	
	
	\bib{chMM}{article}{
   author={Dow, Alan},
   author={Eisworth, Todd},
   title={CH and the Moore-Mrowka problem},
   journal={Topology Appl.},
   volume={195},
   date={2015},
   pages={226--238},
   issn={0166-8641},
   review={\MR{3414886}},
   doi={10.1016/j.topol.2015.09.025},
}
	
	\bib{DowCohen}{article}{
   author={Dow, Alan},
   title={Compact spaces of countable tightness in the Cohen model},
   conference={
      title={Set theory and its applications},
      address={Toronto, ON},
      date={1987},
   },
   book={
      series={Lecture Notes in Math.},
      volume={1401},
      publisher={Springer, Berlin},
   },
   date={1989},
   pages={55--67},
   review={\MR{1031765}},
   doi={10.1007/BFb0097331},
}
				
	
\bib{Cclosed}{article}{
   author={Dow, A.},
   title={Compact C-closed spaces need not be sequential},
   journal={Acta Math. Hungar.},
   volume={153},
   date={2017},
   number={1},
   pages={1--15},
   issn={0236-5294},
   review={\MR{3713559}},
   doi={10.1007/s10474-017-0739-x},
}	


\bib{Fedorchuk}{article}{
   author={Fedor\v{c}uk, V. V.},
   title={Completely closed mappings, and the compatibility of certain
   general topology theorems with the axioms of set theory},
   language={Russian},
   journal={Mat. Sb. (N.S.)},
   volume={99 (141)},
   date={1976},
   number={1},
   pages={3--33, 135},
   review={\MR{0410631}},
}
		

 \bib{HJ73}{article}{
   author={Hajnal, A.},
   author={Juh\'{a}sz, I.},
   title={On hereditarily $\alpha $-Lindel\"{o}f and $\alpha $-separable spaces.
   II},
   journal={Fund. Math.},
   volume={81},
   date={1973/74},
   number={2},
   pages={147--158},
   issn={0016-2736},
   review={\MR{336705}},
   doi={10.4064/fm-81-2-147-158},
}

\bib{Ismail}{article}{
   author={Ismail, Mohammad},
   author={Nyikos, Peter},
   title={On spaces in which countably compact sets are closed, and
   hereditary properties},
   journal={Topology Appl.},
   volume={11},
   date={1980},
   number={3},
   pages={281--292},
   issn={0166-8641},
   review={\MR{585273}},
   doi={10.1016/0166-8641(80)90027-9},
}
		
	
 \bib{Jech2}{book}{
   author={Jech, Thomas},
   title={Set theory},
   series={Springer Monographs in Mathematics},
   note={The third millennium edition, revised and expanded},
   publisher={Springer-Verlag, Berlin},
   date={2003},
   pages={xiv+769},
   isbn={3-540-44085-2},
   review={\MR{1940513 (2004g:03071)}},
}
 \bib{SandL}{article}{
   author={Juh\'{a}sz, I.},
   title={A survey of $S$- and $L$-spaces},
   conference={
      title={Topology, Vol. II},
      address={Proc. Fourth Colloq., Budapest},
      date={1978},
   },
   book={
      series={Colloq. Math. Soc. J\'{a}nos Bolyai},
      volume={23},
      publisher={North-Holland, Amsterdam-New York},
   },
   date={1980},
   pages={675--688},
   review={\MR{588816}},
}
 
 \bib{Mitchell}{article}{
   author={Mitchell, William},
   title={Aronszajn trees and the independence of the transfer property},
   journal={Ann. Math. Logic},
   volume={5},
   date={1972/73},
   pages={21--46},
   issn={0003-4843},
   review={\MR{313057}},
   doi={10.1016/0003-4843(72)90017-4},
}
 
 \bib{Ostaszewski}{article}{
   author={Ostaszewski, A. J.},
   title={On countably compact, perfectly normal spaces},
   journal={J. London Math. Soc. (2)},
   volume={14},
   date={1976},
   number={3},
   pages={505--516},
   issn={0024-6107},
   review={\MR{438292}},
   doi={10.1112/jlms/s2-14.3.505},
}
	
 
 \bib{RudinS}{article}{
   author={Rudin, Mary Ellen},
   title={A normal hereditarily separable non-Lindel\"{o}f space},
   journal={Illinois J. Math.},
   volume={16},
   date={1972},
   pages={621--626},
   issn={0019-2082},
   review={\MR{309062}},
}
 
 \bib{properbook}{book}{
   author={Shelah, Saharon},
   title={Proper forcing},
   series={Lecture Notes in Mathematics},
   volume={940},
   publisher={Springer-Verlag, Berlin-New York},
   date={1982},
   pages={xxix+496},
   isbn={3-540-11593-5},
   review={\MR{675955}},
}
		
  
 \bib{stevoSspace}{article}{
   author={Todor\v{c}evi\'{c}, Stevo},
   title={Forcing positive partition relations},
   journal={Trans. Amer. Math. Soc.},
   volume={280},
   date={1983},
   number={2},
   pages={703--720},
   issn={0002-9947},
   review={\MR{716846}},
   doi={10.2307/1999642},
}
				

\end{biblist}
\end{bibdiv}
\end{document}